\documentclass[11pt]{article}

\newcommand{\BIGOP}[1]{\mathop{\mathchoice%
{\raise-0.22em\hbox{\huge $#1$}}%
{\raise-0.05em\hbox{\Large $#1$}}{\hbox{\large $#1$}}{#1}}}
\newcommand{\bigtimes}{\BIGOP{\times}}
\newcommand{\BIGboxplus}{\mathop{\mathchoice%
{\raise-0.35em\hbox{\huge $\boxplus$}}%
{\raise-0.15em\hbox{\Large $\boxplus$}}{\hbox{\large $\boxplus$}}{\boxplus}}}

\usepackage{latexsym,amsfonts,amssymb,epsfig,tabularx,amsthm,dsfont,enumerate}
\usepackage{amsfonts}
\usepackage{amscd}
\usepackage{color}
\usepackage{amsmath}
\usepackage{amsmath}
\usepackage{float}
\usepackage{amsthm}
\usepackage{amssymb}
\usepackage{mathrsfs}
\usepackage{amstext}
\usepackage{graphicx}
\usepackage{tikz}
\usepackage{verbatim}
\usepackage{framed}

\theoremstyle{definition}

\numberwithin{equation}{section}

\newtheorem{thm}{Theorem}[section]
\newtheorem{rem}[thm]{Remark}
\newtheorem{lem}[thm]{Lemma}
\newtheorem{prop}[thm]{Proposition}
\newtheorem{cor}[thm]{Corollary}

\newtheorem{defi}[thm]{Definition}

\newcommand{\cf}{\ensuremath{\mathcal F}}

\newcommand{\tor}{\ensuremath{\mathbb{T}^d}}
\newcommand{\Z}{\mathbb{Z}}
\newcommand{\N}{\mathbb{N}}
 
\newcommand{\R}{\mathbb{R}}
\newcommand{\C}{\mathbb{C}}

\newcommand{\abss}[1]{\big\vert #1 \big\vert}	

\renewcommand{\a}{\alpha}

\renewcommand{\phi}{\varphi}

\newcommand{\eps}{\varepsilon}

\newcommand{\wt}{\widetilde}
\newcommand\dint{{\,\rm d}}
\newcommand{\supp}{{\rm supp}}

\renewcommand{\S}{\mathcal{S}}
\newcommand{\F}{\mathcal{F}}
\newcommand{\Fd}{\mathbf{F}_d}

\newcommand{\Bspt}{\ensuremath{\mathbf{B}^s_{p,\theta}}}
\newcommand{\Fspt}{\ensuremath{\mathbf{F}^s_{p,\theta}}}
\newcommand{\Fsptw}{\ensuremath{\mathbf{F}^s_{p,2}}}
\newcommand{\Wsp}{\ensuremath{\mathbf{W}^s_p}}

\newcommand{\Bo}{\ensuremath{\mathring{\mathbf B}_{p,\theta}^s}}
\newcommand{\Fo}{\ensuremath{\mathring{\mathbf F}_{p,\theta}^s}}
\newcommand{\Wo}{\ensuremath{\mathring{\mathbf W}_p^s}}

\newcommand{\A}{\ensuremath{\mathbf{A}}}
\newcommand{\B}{\ensuremath{\mathbf{B}}}

\newcommand{\Aspt}{\ensuremath{\mathbf{A}^s_{p,\theta}}}
\newcommand{\Ao}{\ensuremath{\mathring{\mathbf A}_{p,\theta}^s}}

\newcommand{\bproof}{\begin{proof}}
\newcommand{\eproof}{\end{proof}}

\newcommand{\be}{\begin{equation}}
\newcommand{\ee}{\end{equation}}
\newcommand{\beq}{\begin{eqnarray}}
\newcommand{\beqq}{\begin{eqnarray*}}
\newcommand{\eeq}{\end{eqnarray}}
\newcommand{\eeqq}{\end{eqnarray*}}

\newcommand{\bes}{\begin{split}}
\newcommand{\es}{\end{split}}

\title{Change of variable in spaces of mixed smoothness and numerical integration of multivariate functions on the
unit cube} 

\author{ 
Van Kien Nguyen
\and
Mario Ullrich
\and 
Tino Ullrich
}

\begin{document}
\maketitle  
 
\begin{abstract} 
In a recent article by two of the present authors it turned out that Frolov's cubature
formulae are optimal and universal for various settings (Besov-Triebel-Lizorkin spaces) of functions with dominating
mixed smoothness. Those cubature formulae go well together with functions supported inside the unit cube
$[0,1]^d$. The question for the optimal numerical integration of multivariate functions with non-trivial boundary data,
in particular non-periodic functions, arises. In this paper we give a general result that the asymptotic rate of the
minimal worst-case integration error is not affected by
boundary conditions in the above mentioned spaces. In fact, we propose two tailored modifications of Frolov's cubature
formulae suitable for functions supported on the cube (not in the cube) which provide the same minimal
worst-case error
up to a constant. This constant involves the norms of a ``change of variable'' and a ``pointwise multiplication''
mapping, respectively, between the function spaces of
interest. In fact, we complement, extend and improve classical results by Bykovskii, Dubinin and Temlyakov on the
boundedness of change of variable mappings in Besov-Sobolev spaces of mixed smoothness. Our proof technique relies
on a new characterization via integral means of mixed differences and maximal function
techniques, general enough to treat Besov and Triebel-Lizorkin spaces at once. 
The second modification, which only tackles the case of periodic functions, 
is based on a pointwise multiplication and is 
therefore most likely more suitable for applications than the (traditional) 
``change of variable'' approach.
These new theoretical insights are expected to be useful for the design of new (and robust) 
cubature rules for multivariate functions on the cube. 
\end{abstract} 
%

\section{Introduction}
The seminal papers by Korobov \cite{Ko59}, Hlawka \cite{Hl62} and Bakhvalov \cite{Ba63} started the research
for the efficient computation of multivariate integrals $\int_{\Omega}f(x)\,dx$ via cubature formulae of type
\begin{equation}\label{f01}
      Q_n(f) \,:=\, \sum\limits_{i=1}^n \lambda_if(x^i)\,,
\end{equation}
where $X_n:=\{x^1,...,x^n\} \subset \Omega \subset \R^d$ denotes the set of given integration nodes and
$(\lambda_1,...,\lambda_n)$ denotes the vector of integration weights. This field developed into several directions and
attracted a lot of interest in the past 50 years until present, see Temlyakov \cite{Te86,Te90,Te91,Te03}, Dubinin
\cite{Du, Du2}, Skriganov~\cite{Sk94}, Triebel~\cite{Tr10}, Hinrichs et al.\
\cite{Hi10,HiMaOeUl14}, Novak and Wo\'zniakowski~\cite{NW10}, Novak, Krieg \cite{KrNo15},  
Dick and Pillichshammer~\cite{DP}, D\~ung, Ullrich \cite{DU14}, and Markhasin~\cite{Ma13} to mention just a few.
However, several fundamental problems, for instance the construction of optimal (and easy to generate) point sets $X_n$
in \eqref{f01} is still subject of intense research. Frolov's construction \cite{Fr76} dates back to the 1970s and
gives a relatively easy construction of such point sets which perform optimal for several function classes of mixed
smoothness. The recent results in \cite{UU-frolov} will be the starting point for the investigations in this
paper. 

The concept of {\em dominating mixed smoothness} does not only connect discrepancy theory and
optimal numerical integration. Function classes built upon this 
type of multivariate regularity also play an important role in many real-world problems. Several applications are
modeled in Sobolev spaces with dominating mixed smoothness $H^s_{\text{mix}}$, see for instance Yserentant's book
\cite{Y2010} for regularity properties of solutions of the electronic Schr\"odinger equation. However, the Hilbert space
model is not sufficient when it comes to numerical integration issues. We observed in \cite{UU-frolov, HiMaOeUl14} that
it seems to be more appropriate to define smoothness function classes on $L_p$ instead of $L_2$. The case $p=1$ turns out to
be particularly important. To see this let us consider multi-variate {\em kink functions} which often occur in
mathematical finance, e.g., the pricing of a European call option, whose pay-off function possesses a kink at the strike
price~\cite{glasserman2004monte}. 
In general, one can not expect Sobolev regularity higher than $s=3/2$. 
However, when considering Besov regularity we can achieve smoothness $s=2$. 
In a sense, one sacrifices integrability for gaining regularity. Looking at the error bounds in \cite{UU-frolov,
HiMaOeUl14} one observes that certain cubature rules benefit from higher Besov regularity while the integrability $p$
does not enter the picture. 

In the recent paper \cite{UU-frolov} two of the present authors studied the performance of the classical Frolov
cubature formulae \cite{Fr76, MU14} in several settings of classes with dominating mixed smoothness. It
turned out that in
Besov ($\Bo$) and Triebel-Lizorkin ($\Fo$) spaces with dominating mixed smoothness $\mathbf{F}_d := \Ao$, where $\mathbf{A} \in
\{\mathbf{B},\mathbf{F}\}$, of functions supported strictly inside the unit cube $\Omega:=[0,1]^d$, the above mentioned
cubature
formulae provide the optimal rate of convergence among all cubature formulae of type \eqref{f01}
with respect to the minimal worst-case integration error given by 
\begin{equation}\label{eq:minimal}
  \mbox{Int}_n(\Fd) \,:=\, \inf_{Q_n}\,
	e(Q_n,\Fd)\,,\quad n\in \N\,,
\end{equation}
where the infimum is taken over all cubature formulae of the form \eqref{f01} 
and 
\[
e(Q_n,\Fd) \,:=\, \sup\limits_{\|f\|_{\mathbf{F}_d} \leq 1}
	|I(f)-Q_n(f)|\,,\quad n\in \N\,,
\]
for a class $\mathbf{F}_d \subset C(\R^d)$ with $I(f) := \int_{[0,1]^d} f(x)\,dx$.
However, Frolov's method does not seem to be suitable for functions with
non-trivial boundary data, like
multivariate periodic functions $\mathbf{F}_d:=\Aspt(\mathbb{T}^d)$ or non-periodic functions on the cube
$\mathbf{F}_d:=\Aspt([0,1]^d)$, cf.\ Definitions \ref{def:besov}, \ref{def:TL} below. The case of periodic Besov
spaces of dominating mixed smoothness has been studied in \cite{Du, Du2, DU14, HiMaOeUl14}. What concerns optimal
cubature in non-periodic function classes, which are actually most relevant for applications, there are only partial
results known, see for instance \cite{Tr10, DU14} for the correct lower bounds of $\mbox{Int}_n(\Bspt)$. 

A classical approach towards upper bounds has
been proposed by Bykovskii \cite{By85} by first performing a change of
variable to obtain
\begin{equation}\label{f02}
  \int_{[0,1]^d}f(x)\,dx = \int_{[0,1]^d} |\det \psi'(x)| f(\psi(x)) \,dx
\end{equation}
for some differentiable $\psi:[0,1]^d\mapsto[0,1]^d$, 
and use afterwards a cubature formula \eqref{f01} for the right-hand integrand in
\eqref{f02}. The main observation is the fact that this approach results in a modified cubature formula 
\begin{equation}\label{f03}
    Q_n^\psi(f) \,:=\, Q_n\bigl(|\det \psi'|\cdot f\circ\psi\bigr) 
		\,=\, \sum\limits_{i=1}^n \lambda_i\, |\det\psi'(x^i)|\, f(\psi(x^i))\,.
\end{equation}
Arranging the kernel $\psi$ in a way such that the function $|\det \psi'(x)| f(\psi(x))$ preserves mixed smoothness
properties and is supported in $[0,1]^d$ one can now use the above mentioned (optimal) methods to approximate the
integral. Thus, the remaining step in analyzing the method \eqref{f03} is to prove the preservation of mixed
smoothness properties under change of variable, or in other words, to establish the boundedness of mappings 
\begin{equation}\label{f03-1}
 \begin{split}
    T_d^{\psi}:\Aspt &\to \Aspt \\
    f &\mapsto |\det \psi'(x)|f(\psi(x))\,,
 \end{split}
\end{equation}
for suitably chosen $\psi(x)$. A straight-forward choice is certainly given by assuming a tensor product structure. 
That is, e.g., given univariate functions $\psi:\R\to \R$ which are
integrated $C^k(\R)$ bump functions $\varphi$ supported in $[0,1]$ and, 
with a slight abuse of notation, 
define $\psi(x_1,...,x_d):=\psi(x_1)\cdot...\cdot \psi(x_d)$. A natural and simple choice is given by the family of polynomials
\begin{equation}\label{psin}
    \psi_k(t) = \left\{\begin{array}{rcl}
                             \int_0^t \xi^k(1-\xi)^k\,d\xi / \int_0^1 \xi^k(1-\xi)^k \,d\xi &:& t\in [0,1],\\
                             1&:& t>1,\\
                             0&:& t<0\,,
                          \end{array}\right.
\end{equation}
$k\in \N$ and we define 
\begin{equation}\label{tensorpsi}
\psi_k(x) \,=\, \psi_k(x_1)\cdot...\cdot \psi_k(x_d).
\end{equation}
It has been proved by Bykovskii \cite{By85} that $T_d^{\psi_k}$ is bounded in
the Sobolev space with dominating mixed smoothness $\mathbf{W}^s_2$ for $s\in \N$ if $m\geq 2\lfloor s \rfloor+1$ (note
that this
is the same space as $H^s_{\text{mix}}$ mentioned above). This result
has been extended by Temlyakov, see \cite[Thm.\ IV.4.1]{Te93}, to Sobolev spaces $\mathbf{W}^s_p$, $s\in \N$,
$1<p<\infty$, under the condition 
\begin{equation}\label{expl}
      k\geq \Big\lfloor \frac{sp}{p-1}\Big \rfloor+1\,.
\end{equation}
Under the same condition, Temlyakov \cite[Lem.\ IV.4.9]{Te93} showed the boundedness of
$T^{\psi_k}_d:\mathbf{B}^s_{p,\infty}\to \mathbf{B}^s_{p,\infty}$ in the H\"older-Nikol'skij spaces if
$1<p\leq \infty$
and $s>1$. 
Later this was extended by Dubinin \cite{Du, Du2} to the Besov spaces $\Bspt$ if 
$1\leq p,\theta \leq\infty$ and $s>1/p$\,, and by Temlyakov~\cite{Te03} to the spaces 
$\mathbf{W}^s_p$, $s\in \N$, $1\le p\le\infty$. Dubinin and Temlyakov analyzed the change of variable with a
$C^{\infty}$-kernel instead of \eqref{psin}, namely 
\begin{equation}\label{f05_1}
    \psi(t) = \left\{\begin{array}{rcl}
                             \int_0^t e^{-\frac{1}{\xi(1-\xi)}}\,d\xi / \int_0^1 e^{-\frac{1}{\xi(1-\xi)}} \,d\xi &:& t\in [0,1],\\
                             1&:& t>1,\\
                             0&:& t<0\,.
                          \end{array}\right.
\end{equation}
With this particular choice both authors were able to incorporate even the important case $p=1$ (see the discussion
above). Note, that the boundedness is still open for $\mathbf{F}^s_{1,\theta}$, see
Remark~\ref{s>1mixed}, (iii) below. However, the kernel in \eqref{f05_1} has some obvious disadvantages. It involves
terms which get very small and others that get very large at the same time. Therefore, motivated from
numerical robustness issues one is interested in as simple as possible change of variable kernels with the least
(smoothness) requirements. The functions in \eqref{psin} are polynomials and therefore good for our purpose. 
However, from a \emph{tractability} point of view, i.e., if one is interested in the dependence 
of the involved constants on the dimension, every method of this type will be most likely 
intractable. Some results on this issue in the Hilbert space setting can be found in 
Kuo, Sloan and Wo\'zniakowski~\cite{KSW07}, see also Nuyens and Cools~\cite{NC10} for 
numerical experiments.

Our main intention is to extend the mentioned boundedness results on the operators $T^{\psi_k}_d$ to spaces
$\Aspt$, in particular Triebel-Lizorkin spaces of
dominating mixed smoothness $\Fspt$, for the largest possible range of parameters.

\subsection{Contribution and main results} 
The main goal of the present paper is to construct cubature formulae that

\begin{itemize}
  \item are simple, (to some extent) stable and universal (in the sense that the 
  algorithm does not depend on the specific parameters of the spaces)
  \item and that perform optimal in the sense of the minimal worst-case error \eqref{eq:minimal} with respect to
   non-periodic function spaces $\Aspt$ on the unit $d$-cube.
\end{itemize}
The ingredient \eqref{psin} is almost what we want to put into \eqref{f03}, provided we can get rid of the condition
\eqref{expl}. This is possible as we will show in the theorems below. We are able to state the result for spaces
$\Aspt$ on the unit
$d$-cube, where the smoothness (or the parameter $k$ in \eqref{psin}) only depends on the smoothness $s$ of the space
(and not on the integrability $p$). Note, that the result below is not just a
generalization of the results from \cite{By85,Du2,Te93,Te03}. It essentially improves on the stated classical
results for $\Wsp$ and $\mathbf{B}^s_{p,\theta}$ what concerns the choice of the change of variable kernel, which is
an essential ingredient for the resulting modified cubature formula. 

An important special case of our main result in 
Theorem~\ref{changemixed-B} below gives the desired result for the Besov spaces 
$\Bspt$ of dominating mixed smoothness. 

\begin{thm}\label{t2} Let $\Omega=[0,1]^d$, $1 \leq p \leq \infty$, $0< \theta \leq \infty$ 
and $s>0$. Let further $k>\lfloor s\rfloor+2$ ($k>\lfloor s\rfloor+3$ if $p=1$) 
and $\psi_k$ as in \eqref{tensorpsi}. Then,
\begin{itemize}
\item[(i)] $T_d^{\psi_k}:\Bspt \to \Bspt$ is a bounded mapping, and 
\item[(ii)] provided that $\Bspt\subset C(\R^d)$, a corresponding modified cubature formula
\eqref{f03} on $\Bspt$ does not perform asymptotically worse than \eqref{f01} performs on $\Bo$, i.e.,
\[
e(Q_n^{\psi_k},\Bspt) \,\lesssim\, e(Q_n,\Bo)\,,\quad n\in \N\,.
\]
\end{itemize}
\end{thm}

\noindent
The more restrictive condition in the case $p=1$ is actually not necessary, see Remark~\ref{rem:p1}, 
but a proof of this statement would elongate the paper by some pages, so we leave it out.

An analogous results 
for the Triebel-Lizorkin spaces of dominating mixed
smoothness $\mathbf{F}^s_{p,\theta}$ follows from Theorem~\ref{changemixed-F} below. 
Note that we have to exclude the corner points $p=1,\infty$. 

\begin{thm}\label{t1} Let $\Omega=[0,1]^d$, $1<p<\infty$, $1 < \theta\le \infty$ and $s>0$. 
Let further $k>\lfloor s \rfloor+2$ and
$\psi_k$ as in \eqref{tensorpsi}. Then,
\begin{itemize}
\item[(i)] $T_d^{\psi_k}:\Fspt \to \Fspt$ is a bounded mapping, and
\item[(ii)] provided that $\Fspt\subset C(\R^d)$, a corresponding modified cubature formula
\eqref{f03} on $\Fspt$ does not perform asymptotically worse than \eqref{f01} performs on $\Fo$, i.e.,
\[
e(Q_n^{\psi_k},\Fspt) \,\lesssim_{}\, e(Q_n,\Fo)\,,\quad n\in \N\,.
\]
\end{itemize}
\end{thm}

\noindent These results are useful in the sense that they ``transfer'' the optimal rate of the minimal worst-case
error from the known situation $\Ao$ to the more difficult situation $\Aspt$. As a direct consequence we obtain the
equivalence
\begin{equation*}
    \mbox{Int}_n(\mathring{\mathbf A}_{p,\theta}^s) 
		\asymp \mbox{Int}_n({\mathbf A}_{p,\theta}^s(\mathbb{T}^d))
		\asymp \mbox{Int}_n({\mathbf A}_{p,\theta}^s([0,1]^d)) 
		\,,\quad n\in \N\,,
\end{equation*}
in the range of parameters given in Theorems \ref{t2}, \ref{t1} whenever the classes $\Aspt$ consist of continuous
functions. Let us mention the important special case 
\begin{equation*}
    \mbox{Int}_n(\Wo) 
		\asymp \mbox{Int}_n(\Wsp(\mathbb{T}^d))
		\asymp \mbox{Int}_n(\Wsp([0,1]^d)) 
		\,,\quad n\in \N\,,
\end{equation*}
for $1<p<\infty$ and $s>1/p$ taking into account that $\mathbf{F}^s_{p,2} = \Wsp$. Note, that this includes the case of
small smoothness, which is present if $2<p < \infty$ and $1/p<s\leq 1/2$, see \cite{Te91}, \cite{UU-frolov}. 

In case, one is only interested in the numerical integration of periodic functions, 
we present an even simpler modification that satisfies the second point above but 
uses only a pointwise multiplication with a certain function. To our best knowledge this method has not been used
before. In contrast to the
considerations above we will now restrict to function spaces on the $d$-torus $\tor$, see Subsection \ref{sect:torus}
below, i.e., we want to numerically integrate functions which are $1$-periodic in every component. There are many
results for $\Wsp(\tor)$ and $\mathbf{B}^s_{p,\infty}(\tor)$ in this direction, see for instance \cite{Te03} and the
references therein. The recent contribution \cite{HiMaOeUl14} especially deals with periodic Besov spaces of dominating
mixed smoothness and analyzes cubature formulae on digital nets. However, the picture is still not complete. Except
for the Fibonacci cubature rule \cite{Te91, DU14} in $d=2$ the discussed methods in \cite{Tr10, HiMaOeUl14} are either
not optimal or the optimal methods only work for specific (restricted) sets of parameters, like for instance small
smoothness $s<2$ in \cite{HiMaOeUl14}. In other words, the cubature formula is optimal for low smoothness but is not
able to benefit from higher regularity. One goal of this paper is to get rid of such conditions. To gain this,
we pursue a similar goal as above, namely the
reduction of the problem to the $\Ao$ setting via a tailored transformation. To this end we use a sufficiently
smooth function $\psi:\R^d \to [0,\infty)$ being compactly supported 
($\supp(\psi) \subset \Omega$)
such that 
\begin{equation}\label{eq:one}
\sum\limits_{\ell\in \Z^d} \psi(x+\ell) = 1\quad,\quad x \in \R^d\,.
\end{equation}
Clearly, we need $\Omega\supsetneq[0,1]^d$.
Then for any $1$-periodic integrable function in each component it holds
\beqq
   \int_{\R^d} \psi(x)f(x)\,dx &=& \sum\limits_{\ell \in \Z^d}
\int_{[0,1]^d}\psi(x+\ell)f(x+\ell)\,dx\\
    &=& \int_{[0,1]^d} f(x)\sum\limits_{\ell \in \Z^d} \psi(x+\ell)\,dx =
\int_{[0,1]^d} f(x)\,dx\,.
\eeqq
Starting with a cubature formula $Q_n$ of type \eqref{f01} for functions with support in $\Omega$ we
modify as
follows 
\begin{equation}\label{f08}
  \wt Q_n^\psi:=Q_n(\psi f) = \sum\limits_{i=1}^n \psi(x^i)\lambda_i f(\{x^i\})\,,
\end{equation}
where $\{x\}=x-\lfloor x\rfloor$ has to be understood component-wise. In contrast to \eqref{f03-1} the remaining step for analyzing
\eqref{f08} on $1$-periodic function spaces is to establish the boundedness of pointwise multiplier mappings
\begin{equation}\label{point_mult}
 \begin{split}
    \wt T_d^{\psi}:\Aspt(\tor) &\to \Aspt(\R^d) \\
    f(x) &\mapsto \psi(x)f(x)\,.
 \end{split}
\end{equation}
There is a rich theory concerning pointwise multiplication for isotropic spaces $B^s_{p,\theta}(\R^d)$,
$F^s_{p,\theta}(\R^d)$, see Runst, Sickel \cite{RuSi96}. To our best knowledge the field is rather open for
spaces of dominating mixed smoothness. In contrast to the theorems above on change of variable we are able to include
the quasi-Banach case
($\min\{p,\theta\}<1$) within this framework. This is why the natural lower smoothness restrictions enter, namely
$s>\sigma_p$ and $s>\sigma_{p,\theta}$, respectively, where
$$
    \sigma_{p}:= \max\Big\{\frac{1}{p}-1,0\Big\} \quad \mbox{and}\quad
\sigma_{p,\theta}:=\max\Big\{\frac{1}{p}-1,\frac{1}{\theta}-1,0\Big\}\,.
$$

\noindent For the $\mathbf{B}$-spaces the result looks as follows.

\begin{thm}\label{pointmult-B} Let $0< p,\theta \leq \infty$, $s>\sigma_p$
and $\psi\in C^k_0(\Omega)$ with $k\geq \lfloor s \rfloor+1$. 
Then,
\begin{itemize}
 \item[(i)] $\wt T^{\psi}_d : \Bspt(\tor) \to \Bspt$ is a bounded mapping, and
 \item[(ii)] provided that $\Bspt\subset C(\R^d)$ and $\psi$ satisfies~\eqref{eq:one}, 
	a corresponding modified cubature formula
\eqref{f08} on $\Bspt(\tor)$ does not perform asymptotically worse
than \eqref{f01} performs on $\Bo(\Omega)$, i.e.,
\[
e(\wt Q_n^\psi,\Bspt(\tor)) \,\lesssim\, e(Q_n,\Bo(\Omega))\,,\quad n\in \N\,.
\]
\end{itemize}
\end{thm}

\noindent The result for the $\mathbf{F}$-spaces is similar, but we have to exclude the case $p=\infty$.

\begin{thm}\label{pointmult-F} Let $0<p<\infty$, $0<\theta \leq
\infty$, $s>\sigma_{p,\theta}$ and $\psi\in C^k_0(\Omega)$ with $k\geq \lfloor s \rfloor+1$. 
Then,
\begin{itemize}
 \item[(i)] $\wt T^{\psi}_d : \Fspt(\tor) \to \Fspt$ is a bounded mapping, and
 \item[(ii)] provided that $\Fspt\subset C(\R^d)$ and $\psi$ satisfies~\eqref{eq:one}, 
a corresponding modified cubature formula \eqref{f08} on $\Fspt(\tor)$ does not perform asymptotically
 worse than \eqref{f01} performs on $\Fo(\Omega)$, i.e.,
\[
e(\wt Q_n^\psi,\Fspt(\tor)) \,\lesssim\, e(Q_n,\Fo(\Omega))\,,\quad n\in \N\,.
\]
\end{itemize}
\end{thm}

\noindent Both theorems immediately imply the two-sided relation 
\begin{equation}\label{per_circle}
 \mbox{Int}_n(\mathring{\mathbf A}_{p,\theta}^s) \asymp
\mbox{Int}_n({\mathbf A}_{p,\theta}^s(\mathbb{T}^d))\,,\quad n\in \N\,,
\end{equation}
in the above range of parameters if additionally $\Aspt \hookrightarrow
C(\R^d)$ holds true. 

\subsection{Organization of the paper}
The paper is organized as follows. In Section \ref{sec:prel} we collect the required tools from Fourier analysis,
especially maximal inequalities. Section \ref{sec:spaces} is devoted to the definition and properties of the function
spaces under consideration. There we characterize Besov-Triebel-Lizorkin spaces by mixed iterated differences
which turn out to be the crucial tool for proving our main results. Sections \ref{sec:proof_modi} and \ref{sec:mult}
are the heart of the paper in which we prove the change of variable and pointwise multiplication results already
discussed in the introduction section. For the convenience of the reader we also discuss the one-dimensional versions
of the results in Sections \ref{sec:proof_modi} and \ref{sec:mult} in detail. 

\subsection{Notation} 
As usual $\N$ denotes the natural numbers, $\N_0=\N\cup\{0\}$, 
$\Z$ denotes the integers, 
$\R$ the real numbers, 
and $\C$ the complex numbers. By $\mathbb{T}:=\R/\Z$ we denote the torus represented with the interval $[0,1]$ in
$\R$\,. The letter $d$ is always reserved for the underlying dimension in $\R^d, \Z^d$ etc and by $[d]$ we mean
$[d]=\{1,...,d\}$. We denote
with $\langle x,y\rangle$ or $x\cdot y$ the usual Euclidean inner product in $\R^d$ or $\C^d$.
For $0<p\leq \infty$ and $x\in \R^d$ we denote $|x|_p = (\sum_{i=1}^d |x_i|^p)^{1/p}$ with the
usual modification in the case $p=\infty$. 
By $(x_1,\ldots,x_d)>0$ we mean that each coordinate is positive. As usual we decompose $x\in \R$ in $x = \lfloor x
\rfloor + \{x\}$ where $0\leq \{x\} <1$ (resp.~$\{x\}$ for $x\in \R^d$ is meant component-wise). 

If $X$ and $Y$ are two (quasi-)normed spaces, the (quasi-)norm
of an element $x$ in $X$ will be denoted by $\|x\|_X$. 
The symbol $X \hookrightarrow Y$ indicates that the
identity operator is continuous. For two sequences $a_n$ and $b_n$ we will write $a_n \lesssim b_n$ if there exists a
constant $c>0$ such that $a_n \leq c\,b_n$ for all $n$. We will write $a_n \asymp b_n$ if $a_n \lesssim b_n$ and $b_n
\lesssim a_n$.

\goodbreak

\section{Tools from Fourier analysis}
\label{sec:prel}

In this section we will collect the required tools from Fourier analysis. 
\subsection{Preliminaries}
\noindent
Let $L_p=L_p(\R^d)$, $0 < p\le\infty$, be the space of all functions 
$f:\R^d\to\C$ such that 
\[
\|f\|_p := \Big(\int_{\R^d} |f(x)|^p \dint x \Big)^{1/p} < \infty
\]
with the usual modification if $p=\infty$. 
We will also need $L_p$-spaces on compact domains $\Omega\subset\R^d$ 
instead of $\R^d$. 
We write $\|f\|_{L_p(\Omega)}$ for the corresponding (restricted) $L_p$-norm. 

For $k\in \N_0$ we denote by $C^{k}_0(\R^d)$ the collection of all compactly supported 
functions $\varphi$ on $\R^d$ which have $D^{\alpha}\varphi$ uniformly continuous on $\R^d$ 
for $|\alpha|_{\infty}\leq k$. 
Additionally, we define the spaces infinitely differentiable functions $C^{\infty}(\R^d)$ and infinitely differentiable
functions with compact support $C_0^{\infty}(\R^d)$ as well as the \emph{Schwartz space} $\S=\S(\R^d)$ of 
all rapidly decaying infinitely differentiable functions on $\R^d$, i.e.,  
\[
\S := \bigl\{\varphi\in C^{\infty}(\R^d)\colon \|\varphi\|_{k,\ell}<\infty 
\;\text{ for all } k,\ell\in\N\bigr\}\,,
\]
and
$$
    \|\varphi\|_{k,\ell}:=\Big\|(1+|\cdot|)^k\sum_{|\alpha|_1\leq \ell}|D^{\alpha}\varphi(\cdot)|\Big\|_{\infty},\ \ \alpha\in \N_0^d\,.
$$

The space $\mathcal{S}'(\R^d)$, the topological dual of $\mathcal{S}(\R^d)$, is also referred to as the set of tempered
distributions on $\R^d$. Indeed, a linear mapping $f:\mathcal{S}(\R^d) \to \C$ belongs
to $\mathcal{S}'(\R^d)$ if and only if there exist numbers $k,\ell \in \N$
and a constant $c = c_f$ such that
\beqq
    |f(\varphi)| \leq c_f\|\varphi\|_{k,\ell}
\eeqq
for all $\varphi\in \mathcal{S}(\R^d)$. The space $\mathcal{S}'(\R^d)$ is
equipped with the weak$^{\ast}$-topology.

For $f\in L_1(\R^d)$ we define the Fourier transform
\[
\F f(\xi) 
\,=\, \int_{\R^d} f(y) e^{-2\pi i \langle \xi, y\rangle} d y, \qquad \xi\in\R^d,  
\]
and the corresponding inverse Fourier transform $\F^{-1}f(\xi)=\F f(-\xi)$.
As usual, the Fourier transform can be extended to $\mathcal{S}'(\R^d)$
by $(\cf f)(\varphi) := f(\cf \varphi)$, where
$\,f\in \mathcal{S}'(\R^d)$ and $\varphi \in \mathcal{S}(\R^d)$. 
The mapping $\cf:\S'(\R^d) \to \S'(\R^d)$ is a bijection.

The convolution $\varphi\ast \psi$ of two
square-integrable
functions $\varphi, \psi$ is defined via the integral
\begin{equation*}
    (\varphi \ast \psi)(x) = \int_{\R^d} \varphi(x-y)\psi(y)\,dy\,.
\end{equation*}
If $\varphi,\psi \in \mathcal{S}(\R^d)$ then $\varphi \ast \psi$ still belongs to
$\mathcal{S}(\R^d)$. 
In fact, we have $\varphi \ast \psi\in\S(\R^d)$ even if $\varphi \in \mathcal{S}(\R^d)$ 
and $f\in L_1(\R^d)$. The convolution can be extended to $\mathcal{S}(\R^d)\times \mathcal{S}'(\R^d)$ via
$(\varphi\ast f)(x) = f(\varphi(x-\cdot))$. It makes sense pointwise and is 
a $C^{\infty}$-function in $\R^d$.
\subsection{Maximal inequalities}

Let us provide here the maximal inequalities for the Hardy-Littlewood and Peetre maximal functions, respectively, which
are essential for the proofs of Theorems \ref{changemixed-B} and \ref{changemixed-F}. For further details we refer to \cite[1.2, 1.3]{Vyb06} or \cite[Chapt.\ 2]{ST}\,.

For a locally integrable function $f:\R^d\to \C$ we denote by $Mf(x)$ the Hardy-Littlewood maximal function defined by
\begin{equation}
  (Mf)(x) = \sup\limits_{x\in Q} \frac{1}{|Q|}\int_{Q}\,|f(y)|\,dy\quad,\quad x\in\R^d \label{maxfunc}
  \quad,
\end{equation}
where the supremum is taken over all cubes with sides parallel to the coordinate axes containing~$x$.
A vector valued generalization of the classical Hardy-Littlewood maximal inequality is due to
Fefferman and Stein \cite{FeSt71}.
\begin{thm}\label{feffstein}\rm For $1<p<\infty$ and $1 <\theta \leq \infty$ there exists a constant $c>0$, such that
  \begin{equation*}
      \Big\|\Big(\sum\limits_{\ell \in I} |M f_{\ell}|^\theta\Big)^{1/\theta}\Big\|_p \leq c
      \Big\|\Big(\sum\limits_{\ell \in I} |f_{\ell}|^\theta\Big)^{1/\theta}\Big\|_p 
  \end{equation*}
  holds for all sequences $\{f_\ell\}_{\ell\in I}$ of locally Lebesgue-integrable functions on $\R^d$.
\end{thm}
\noindent We require a direction-wise version of (\ref{maxfunc})
\begin{equation*}
  (M_i f)(x) = \sup\limits_{s>0}\frac{1}{2s}\int_{x_i-s}^{x_i+s} |f(x_1,...,x_{i-1},t,x_{i+1},...,x_d)|\,dt
  \quad,\quad x\in\R^d. 
\end{equation*}
We denote the composition of these operators by  $M_e=\prod_{i\in e}M_i $, where $e$ is a subset of $[d]$ and
$M_{\ell}M_k$ has to be interpreted as $M_{\ell}\circ M_k$. The following version of the Fefferman-Stein maximal
inequality is due St\"ockert \cite{Stoe}.

\begin{thm}\label{bagby}\rm For $1<p<\infty$ and $1 <\theta \leq \infty$ there exists a constant $c>0$, such that for
any $i\in [d]$
  \begin{equation*}
      \Big\|\Big(\sum\limits_{\ell \in I} |M_i f_{\ell}|^\theta\Big)^{1/\theta}\Big\|_p \leq c
      \Big\|\Big(\sum\limits_{\ell \in I} |f_{\ell}|^\theta\Big)^{1/\theta}\Big\|_p 
  \end{equation*}
  holds for all sequences $\{f_\ell\}_{\ell\in I}$ of locally Lebesgue-integrable functions on $\R^d$.
\end{thm}
Iteration of this theorem yields a similar boundedness property for the operator $M_{[d]}$.\\

\noindent The following construction of a maximal function is due to Peetre, Fefferman and Stein.
Let $b=(b_1,...,b_d)>0$, $a>0$, and $f \in L_1(\R^d)$ with $\F f$ compactly
supported. We define the Peetre maximal function $P_{b,a}f$ by
\begin{equation}
  P_{b,a}f(x) = \sup\limits_{z\in \R^d} \frac{|f(x-z)|}{(1+|b_1z_1|)^a\cdot...\cdot (1+|b_dz_d|)^a}
  \quad.\label{petfefste}
\end{equation}

\begin{lem}\label{maxunglem} Let $\Omega \subset \R^d$ be a compact set. Let further $a>0$
and $\alpha = (\alpha_1,...,\alpha_d) \in \N_0^d$. Then there exist two constants
$c_1,c_2>0$ (independently of $f$) such that
  \begin{equation}
    \begin{split}
      P_{(1,...,1),a}(D^{\alpha}f)(x) &\leq c_1 P_{(1,...,1),a} f(x)\\
      &\leq c_2\big(M_d\big(M_{d-1}\big(...\big(M_1|f|^{1/a}\big)...\big)
      \big)\big)^{a}(x)\quad
    \end{split}
    \label{g1}
  \end{equation}
  holds for all $f\in L_1(\R^d)$ with $\supp(\F f) \subset \Omega$ and all $x\in\R^d$. 
	The constants $c_1$, $c_2$ depend on $\Omega$.
\end{lem}
We will need the vector-valued Peetre maximal inequality which is a direct consequence of Lemma \ref{maxunglem} together
with Theorem \ref{bagby}.
\begin{thm}\label{peetremax} Let $0<p<\infty$, $0<\theta\leq \infty$ and $a>\max\{1/p,1/\theta\}$. Let further
            $b^{\ell} = (b^{\ell}_1,...,b^{\ell}_d)>0$ for $\ell \in I$ and $\Omega = \{\Omega_{\ell}\}_{\ell\in I}$,
            such that
            $$
                       \Omega_{\ell} \subset [-b_1^{\ell},b_1^{\ell}]\times\cdots \times [-b_d^{\ell},b_d^{\ell}]
            $$
            is compact for $\ell\in I$. Then there is a constant $C>0$ (independently of $f$ and $\Omega$) such that 
            $$
               \Big\|\Big(\sum\limits_{\ell \in I} |P_{b^{\ell},a}f_{\ell}|^{\theta}\Big)^{1/\theta}\Big\|_p \leq
C 
	       \Big\|\Big(\sum\limits_{\ell \in I} |f_{\ell}|^{\theta}\Big)^{1/\theta}\Big\|_p
            $$
            holds for all systems $f = \{f_{\ell}\}_{\ell\in I}$ with $\supp(\F f_{\ell}) \subset \Omega_{\ell}$, $\ell
\in I$\,.
 \end{thm}


\section{Function spaces of dominating mixed smoothness}
\label{sec:spaces}

In this section we introduce the function spaces under consideration, namely, the Besov and 
Triebel-Lizorkin spaces of dominating mixed smoothness. 
Note that the Sobolev spaces $\Wsp$ of mixed smoothness appear as a special case of the 
Triebel-Lizorkin spaces, namely $\Wsp = \mathbf{F}^s_{p,2}$ if $1<p<\infty$ and $s>0$. There are several equivalent
characterizations of these spaces, see e.g. \cite{Vyb06}. 
We begin with the usual definition of the spaces using a dyadic decomposition of unity on 
the Fourier side, following Triebel~\cite{Tr10}. Afterwards, we present an equivalent characterization by mixed
iterated differences. To begin with, we introduce the concept of {\em smooth dyadic decomposition of unity}. Let
$\{\eta_j\}_{j\in\N_0}\subset C_0^\infty(\R)$ such that 
\begin{enumerate}
	\item[$(i)$] $\supp(\eta_0)\subset\{t\colon|t|\le 2\}$,
	\item[$(ii)$] $\supp(\eta_j)\subset\{t\colon 2^{j-1}\le |t|\le  2^{j+1}\}$, $j\in \N$,
	\item[$(iii)$] for all $\ell\in\N_0$ it holds $\sup_{t,j} 2^{j\ell}|D^\ell\eta_j(t)|\le c_\ell <\infty$ and 
	\item[$(iv)$] $\sum_{j\in\N_0}\eta_j(t)=1$ for all $t\in\R$\,.
\end{enumerate}
Such a family of functions obviously exists. To approach functions spaces with dominating mixed smoothness we define
the family $\{\eta_j\}_{j\in\N_0^d}$ of $d$-variate functions as the tensor products
\begin{equation*}\label{eq:Psi}
\eta_j(x) \,=\, \prod_{i=1}^d \eta_{j_i}(x_i),
\end{equation*}
where $j=(j_1,\dots,j_d)\in\N_0^d$ and $x=(x_1,\dots,x_d)\in\R^d$.

\subsection{Spaces on $\R^d$} \label{subsec:spaces_R}

Let us start with the definition of the function spaces $\Aspt = \Aspt(\R^d)$ with 
$\A\in\{\mathbf{B},\mathbf{F}\}$ defined on the entire $\R^d$.

\begin{defi}[Besov space] \label{def:besov}
Let $0 < p,\theta\le\infty$, $s\in\R$, and 
$\{\eta_j\}_{j\in\N_0^d}$ be as above. 
The \emph{Besov space of dominating mixed smoothness} 
$\Bspt=\Bspt(\R^d)$ is the set of all $f\in \S'(\R^d)$ 
such that 
\[
\|f\|_{\Bspt} \,:=\, 
\Big(\sum_{j\in\N_0^d} 2^{s |j|_1 \theta}\, 
	\left\|\F^{-1}[\eta_j \F f]\right\|_p^\theta\Big)^{1/\theta} \,<\,\infty
\]
with the usual modification for $\theta=\infty$. 
\end{defi}

\begin{defi}[Triebel-Lizorkin space] \label{def:TL}
Let $0< p<\infty$, $0 < \theta\le\infty$, $s\in\R$, and 
$\{\eta_j\}_{j\in\N_0^d}$ be as above. 
The \emph{Triebel-Lizorkin space of dominating mixed smoothness} 
$\Fspt=\Fspt(\R^d)$ is the set of all $f\in \S'(\R^d)$ 
such that 
\[
\|f\|_{\Fspt} \,:=\, 
\Big\|\Big(\sum_{j\in\N_0^d} 2^{s |j|_1 \theta}\, 
	\left|\F^{-1}[\eta_j \F f]\right|^\theta\Big)^{1/\theta}\Big\|_p \,<\,\infty
\]
with the usual modification for $\theta=\infty$. 
\end{defi}

%
%
%

\begin{rem} In the special case $\theta=2$ and $1<p<\infty$ we put $\Wsp := \Fsptw$ which denotes the 
\emph{Sobolev spaces of dominating mixed smoothness}. It is well-known 
(cf.~\cite[Chapt.\ 2]{ST} for $d=2$), 
that in case $s\in\N_0$ the spaces $\Wsp$ can be equivalently normed by 
\[
\|f\|_{\Wsp} \,\asymp\, \Big(\sum_{\substack{\a\in\N_0^d\\ |\a|_\infty\le s}} \|D^\a f\|_p^p\Big)^{1/p}\,.
\]
 
\end{rem}

\begin{rem} Different choices of smooth dyadic decompositions of unity $\{\eta_j\}_{j\in\N_0}$ 
lead to equivalent (quasi-)norms. 
\end{rem}

\noindent The next lemma collects some frequently used embedding properties of the spaces. 

\begin{lem}\label{emb} Let $0<p<\infty$,  $0<\theta\leq \infty$, $s\in \R$ and 
$\A\in\{\B,\mathbf{F}\}$.
\begin{description}
 \item (i) If $s>\max\{1/p-1,0\}$ then 
 $$
    \A^{s}_{p,\theta} \hookrightarrow L_{\max\{p,1\}}(\R^d)\,.
 $$
 \item (ii) If $s>1/p$ then
 $$
    \A^s_{p,\theta} \hookrightarrow C(\R^d)\,.
 $$
 \end{description}
\end{lem}

\bproof 
For a proof we refer to \cite[Chapt.\ 2]{ST}.\\
\eproof
 
Lemma~\ref{emb}(ii) gives a sufficient condition for 
$\A^s_{p,\theta}$ to consist only of continuous functions. 
This is of particular importance for the analysis of cubature rules since function
evaluations have to be defined reasonably.

\subsection{Spaces on the $d$-torus and non-periodic functions on the cube $[0,1]^d$}
\label{sect:torus}

Let $\tor=[0,1]^d$ and  $f:\tor \to \C$ be a function defined on $\R^d$ and supposed to be 
$1$-periodic in every component. We will define periodicity in a wider sense. A distribution $f\in \S'(\R)$ is a
periodic distribution from $\S'_{\pi}(\R^d)$ if and only if
$$
    f(\varphi) = f(\varphi(\cdot+k))\,,\quad \varphi \in \S(\R^d)\,,\ k\in \Z^d\,.
$$
Now we define
$$
  \A^s_{p,\theta}(\tor) = \{f \in \S'_{\pi}(\R^d)\,:\,\|f\|_{\A^s_{p,\theta}(\tor)}<\infty\}
$$
with
\[
\|f\|_{\Bspt(\tor)} \,:=\, 
\Big(\sum_{j\in\N_0^d} 2^{s |j|_1 \theta}\, 
	\|\F^{-1}[\eta_j \F f]\|_{L_p(\tor)}^\theta\Big)^{1/\theta}
\]
and
\[
\|f\|_{\Fspt(\tor)} \,:=\, 
\Big\|\Big(\sum_{j\in\N_0^d} 2^{s |j|_1 \theta}\, 
	|\F^{-1}[\eta_j \F f]|^\theta\Big)^{1/\theta}\Big\|_{L_p(\tor)},
\]
respectively. Note, that we replace the $L_p(\R^d)$-norms by $L_p(\tor)$-norms. 
If we denote by $\{c_k(f)\}_{k\in\Z^d}$ the Fourier coefficients of $f$, then 
we could write
\[
\F^{-1}[\eta_j \F f](x) \,=\, \sum_{k\in\Z^d} c_k(f)\, \eta_j(k)\, e^{2\pi i\, kx},
\]
see \cite[p.~150]{ST}.

Next we turn to the subspace of functions which are supported in a compact set $\Omega$. 
That is, we define for $s>\sigma_p= \max\{1/p -1,0\}$ the spaces 
\begin{equation}\label{ringspace}
\Ao (\Omega)\,:=\, \bigl\{ f\in \Aspt(\R^d)\colon \supp(f)\subset \Omega\bigr\}\,.
\end{equation}
Due to Lemma~\ref{emb}(i) this definition is reasonable. If $\Omega=[0,1]^d$ we write $\Ao$ instead of $\Ao([0,1]^d)$.
The space $\Ao$ can be interpreted as subspaces of 
$\A^s_{p,\theta}(\tor)$. Finally, we define the spaces $\Aspt([0,1]^d)$. For large enough $s$, 
i.e., $s>\sigma_p$, we define 
$\Aspt([0,1]^d)$ as the collection of all functions $f\in\Aspt$ restricted to the unit cube $[0,1]^d$, i.e.,
\[
\Aspt([0,1]^d) \,:=\, \bigl\{ f=g|_{[0,1]^d}\colon g\in \Aspt(\R^d)\bigr\}.
\]
The (quasi-)norm on these spaces is usually given by 
\[
\|f\|_{\Aspt([0,1]^d)} \,:=\, \inf\big\{\|g\|_{\Aspt}:\, g\in\Aspt\ \text{and}\ g|_{[0,1]^d}=f\big\}.
\]

\subsection{Characterization by mixed iterated differences}\label{subsec:differences}
In this section we will provide a characterization of the above defined spaces which is actually the classical way of
defining them, see for instance Nikol'skij~\cite{Nik75}, Amanov \cite{Am76}, Schmei{\ss}er, 
Triebel \cite{ST}, Temlyakov \cite{Te93} and the references therein. For us it will be convenient to replace the
classical moduli of continuity by so-called rectangular means of differences, see \eqref{rectm} below. 
This is is the counterpart of the ball means of differences for isotropic spaces, 
see~\cite[Thm.~2.5.11]{Tr83}. 

Let us first recall the basic concepts. For univariate functions $f:\R \to \C$ the $m$th difference operator
$\Delta_h^{m}$ is defined by
\begin{equation}\label{expdiff}
\Delta_h^{m}(f,t) := \sum_{j =0}^{m} (-1)^{m - j} \binom{m}{j} f(t + jh)\quad,\quad t\in \R,\ h\in \R\,.
\end{equation}

Let $e$ be any subset of
$[d]=\{1,...,d\}$. For multivariate functions $f:\R^d\to \C$ and $h\in \R^d$
the mixed $(m,e)$th difference operator $\Delta_h^{m,e}$ is defined by
\begin{equation*}
\Delta_h^{m,e} := \
\prod_{i \in e} \Delta_{h_i,i}^m\quad\mbox{and}\quad \Delta_h^{m,\emptyset} =  \operatorname{Id},
\end{equation*}
where $\operatorname{Id}f = f$ and $\Delta_{h_i,i}^m$ is the univariate operator applied to the $i$-th coordinate of $f$
with the other variables kept fixed. Whenever we write the ``product'' $\prod_{i \in e} A_i$ of operators $A_i$ we mean
the concatenation of the $A_i$, see also the small paragraph before Theorem \ref{bagby} above. 

Now we define the so-called rectangular means of differences by
\begin{equation}\label{rectm}
\mathcal{R}_{m}^e(f,t,x):= \int_{[-1,1]^d}|\Delta_{(h_1t_1,...,h_dt_d)}^{m,e}(f,x)|d h,\quad x\in \R^d, t\in
(0,1]^d\,.
\end{equation}
For $j\in \N_0^d$ we put $e(j) = \{i:j_i \neq 0\}$ and $e_0(j)=[d]\backslash e(j)$. We have the following theorems.

\begin{thm}\label{diff} Let $0<p, \theta\leq \infty$ and $s>\sigma_{p}$. Let further $m \in \N$ be a natural number with
$m>s$\,. Then 
$$
    \|f\|_{\Bspt} \asymp \|f\|_{\Bspt}^{(m)}\quad,\quad f\in L_1(\R^d)\,,
$$
where 
\begin{equation*}\label{f4}
     \|f\|_{\Bspt}^{(m)} := \Big(\sum\limits_{j\in \N_{0}^d} 2^{s|j|_1 \theta}\|
\mathcal{R}^{e(j)}_m(f,2^{-j},\cdot)\|_p^{\theta}\Big)^{1/\theta}\,.
\end{equation*}
In case $\theta=\infty$ the sum above 
is replaced by the supremum over $j$. Here 
$2^{-j}:=(2^{-j_1},...,2^{-j_d})$\,.

\end{thm}

\begin{thm}\label{fdiff} Let $0 < p <\infty$, $0< \theta\leq \infty$ and $s>\sigma_{p,\theta}$. Let further $m \in \N$
be a natural number with $m>s$\,. Then 
$$
  \|f\|_{\Fspt} \asymp \|f\|^{(m)}_{\Fspt}\quad,\quad f\in L_1(\R^d)\,,
$$
where 
$$
  \|f\|^{(m)}_{\Fspt}:=\Big\|\Big(\sum\limits_{j\in \N_{0}^d}
2^{s|j|_1\theta}\mathcal{R}^{e(j)}_m(f,2^{-j},\cdot)^{\theta}\Big)^{1/\theta}\Big\|_p\,.
$$
\end{thm}

\medskip

Since the characterizations of the spaces $\Bspt$ and $\Fspt$ that are given in 
Theorems~\ref{diff}~\&~\ref{fdiff} are, as far as we know, not proven in the literature, 
we will give a proof here for the convenience of the reader.

\bigskip
\noindent
{\bf Proof of Theorems~\ref{diff}~\&~\ref{fdiff}.}
The statements in Theorems \ref{diff}~\&~\ref{fdiff} are ``discretized versions'' of the
characterizations given in \cite{Ul06}. In addition, the range  of the sum $\sum_{j\in
\N_0^d}$ is crucial. 
Here, we mainly 
have to establish the
relation $\|f\|_{\Fspt} \asymp \|f\|^{(m)}_{\Fspt}$ which is done via the characterization 
for the spaces $\Bspt$ and $\Fspt$ via local means, see \cite{Vyb06}, \cite{TU08}. 
Given a function $\Psi:\R \to \C$, we denote by
$L_\Psi \in \N$ the maximal number such that $\Psi$ has vanishing moments,
i.e.,
$$
\int_{\R} t^{\alpha} \Psi(t)\,dt = 0\quad,\quad \alpha =
0,...,L_\Psi\,.
$$
Let $\Psi_0 \in S(\R)$ be a function satisfying the following conditions
\be\label{condphi}
\int_{\R} \Psi_0(t)\,dt \not=0 \quad\text{and}\quad L_{\Psi}\geq R\ \ \text{for}\ \
 \Psi(t) = \Psi_0(t)-\frac{1}{2}\Psi_0(t/2)\,,
\ee
for some $R \in \N$. For $k\in \N$ we put $\Psi_{k}(t)=2^k\Psi(2^kt),\  t\in \R$. If $j \in \N_0^d$ we denote
\beqq
   \Psi_{j}(x_1,...,x_d) = \Psi_{j_1}(x_1)\cdot...\cdot
   \Psi_{j_d}(x_d)\,,\ \ x\in\R^d\,.
\eeqq
We have the following proposition, see \cite[Thm. 1.23]{Vyb06} and \cite{TU08}.
\begin{prop}\label{local} Let $s \in \R$ and $\Psi_0$ be given
by \eqref{condphi} with $R+1> s$.\\
{\em (i)} Let $0<p,\theta\leq \infty$. Then
$\Bspt$ is the collection of all $f\in S'(\R^d)$ such
that
\beqq
     \|f\|_{\Bspt} = \Big(\sum_{j\in \N_0^d} 2^{|j|_1s\theta}\|\Psi_{j}\ast
          f\|_p^{\theta}\Big)^{1/\theta} <\infty\,.
\eeqq
{\em (ii)} Let $0<p<\infty$ and $0<\theta\leq \infty$. Then ${\Fspt}$ is the collection of all
$f\in S'(\R^d)$ such that
\beqq
    \|f\|_{\Fspt} =\Big \|\Big(\sum_{j\in \N_0^d }2^{|j|_1s\theta}|\Psi_{j}\ast
     f|^{\theta}\Big)^{1/\theta}\Big\|_p      <\infty\,.
\eeqq
\end{prop}

Now we are in a position to prove Theorem \ref{fdiff}. The proof of Theorem \ref{diff} is similar but technically less
involved.
\begin{proof}[Proof of Theorem~\ref{fdiff}] {\it Step 1.} We prove the inequality
\be\label{in}
\|f\|^{(m)}_{\Fspt}\ \lesssim\ \|f\|_{\Fspt}\,.
\ee
This inequality is a consequence of \cite[Thm.\ 3.4.1]{Ul06}. Indeed, from the equivalent norm
\beqq
\|f\|_{\Fspt} \asymp 
\|f\|_p+\sum_{e\subset [d], e\not=\emptyset} 
\bigg\| \bigg(\int_{(0,\infty)^{|e|}} \Big(\prod_{i\in e}t_i^{-s\theta}\Big) \mathcal{R}^{e}_m(f,t,\cdot)^{\theta}
\prod_{i\in e}\frac{dt_i}{t_i}\bigg)^{1/\theta}\bigg\|_p\,,
\eeqq
see \cite[Thm.\ 3.4.1]{Ul06}, we have
\beqq
\|f\|_{\Fspt} \gtrsim
\|f\|_p+\sum_{e\subset [d], e\not=\emptyset} 
\bigg\| \bigg(\int_{(0,1)^{|e|}} \Big(\prod_{i\in e}t_i^{-s\theta}\Big) \mathcal{R}^{e}_m(f,t,\cdot)^{\theta}
\prod_{i\in e}\frac{dt_i}{t_i}\bigg)^{1/\theta}\bigg\|_p\,.
\eeqq
Discretizing the right-hand side of the above inequality we obtain
\beqq
&&\Big\|\Big(\sum\limits_{j\in \N_{0}^d}
2^{|j|_1s\theta}\mathcal{R}^{e(j)}_m(f,2^{-j},\cdot)^{\theta}\Big)^{1/\theta}\Big\|_p\\
&&\qquad\qquad\lesssim ~\|f\|_p+\sum_{e\subset [d], e\not=\emptyset} 
\bigg\| \bigg(\int_{(0,1)^{|e|}} \Big(\prod_{i\in e}t_i^{-s\theta}\Big) \mathcal{R}^{e}_m(f,t,\cdot)^{\theta}
\prod_{i\in e}\frac{dt_i}{t_i}\bigg)^{1/\theta}\bigg\|_p\,,
\eeqq
which implies \eqref{in}\,.\\
{\it Step 2.} We prove the reverse inequality in \eqref{in}. This time we rely on the intrinsic characterization of
$\Fspt$ via local means, see Proposition \ref{local}.\\
{\it Substep 2.1.} Preparation. According to \cite[3.3.1, 3.3.2]{Tr92}, for $m\in \N$ there
 exist function $\psi_m(t)$ with 
 $$\supp(\psi) \subset
 \Big(-\frac{1}{m},\frac{1}{m}\Big)\qquad \text{ and }\qquad \int_{\R}\psi_m(t)\,dt = 1$$
 such that
\beqq
     \Psi_0(t) = \frac{(-1)^{m+1}}{m!}
     \sum_{u=1}^{m}\sum_{v=1}^{m}(-1)^{u+v}\binom{m}{u}\binom{m}{v}v^{m}(uv)^{-1}\psi\Big(\frac{t}{uv}\Big)~,~t\in\R\,,
\eeqq
 satisfies (\ref{condphi}) with $R = m$. Putting
\beqq
     \eta(t) = \frac{(-1)^{m+1}}{m!}
     \sum_{v=1}^{m}(-1)^{m-v}\binom{m}{v}v^{m}v^{-1}\psi\Big(\frac{t}{v}\Big)~,~t\in
     \R\,,
\eeqq
then we have $\supp(\eta) \in (-1,1)$ and
 moreover
\beqq
     \Psi_0(t) = \sum_{u=1}^{m}(-1)^{m-u}\binom{m}{u}u^{-1}\eta\Big(\frac{t}{u}\Big)\,.
\eeqq
 A simple computation gives for a univariate function $g$
\beq
        \Psi_0*g(t) &=&\int_{\R} \Psi_0(h)g(t-h)\,dh \nonumber\\
        &=& \int_{\R} \eta(-h)\sum_{u=1}^{m}(-1)^{m-u}\binom{m}{u}g(t+uh)\,dh\nonumber\\
 	&=&\int_{-1}^{1} \eta(-h)\big[\Delta^{m}_{h}g(t)-(-1)^{m}g(t)\big]\,dh\,  \label{equi-1}
\eeq
and
\beq
      \Psi_k*g(t)    &=&\int_{\R} \Psi(h)g(t-2^{-k}h)\,dh\nonumber\\
         &=& \int_{\R}
         \Big[\eta(-h)-\frac{1}{2}\eta\Big(-\frac{h}{2}\Big)\Big]\Delta^{m}_{2^{-k}h}g(t)\,dh\,\nonumber\\
&=&\int_{-1}^{1}\eta(-h)\big[\Delta^{m}_{2^{-k}h}-\Delta^{m}_{2^{-k+1}h}\big]g(t)\,dh\,.\label{equi-2}
\eeq
{\it Substep 2.2.} We define the function $\eta(\cdot)$ on $\R^d$ as $d$-fold tensor product, i.e.,
\beqq
\eta(x)=\eta(x_1)\cdot...\cdot\eta(x_d),\ \ x=(x_1,...,x_d)\in \R^d.
\eeqq
From \eqref{equi-1} and \eqref{equi-2} we have for $j\in \N_0^d$
\beqq
\Psi_{j}\ast f(x)= \int_{[-1,1]^d} \eta(-h)\Big[\prod_{i\in
e(j)}\big(\Delta^{m}_{2^{-j_i}h_i,i}-\Delta^{m}_{2^{-j_i+1}h_i,i} \big)\prod_{i\in e_0(j)}\big(\Delta^{m}_{h_i,i}
-(-1)^{m}\text{Id}_i \big)  \Big]f(x)dh\,.
\eeqq
Then we obtain
\beqq
|\Psi_{j}\ast f(x)|&\leq & \sum_{{u}\in \{0,1\}^d} \mathcal{R}^{e(j+u)}_m(f,2^{-j+u},x)
\eeqq
which leads to
\beqq
\Big \|\Big(\sum_{j\in \N_0^d }2^{s|j|_1 \theta}|\Psi_{j}\ast f|^{\theta}\Big)^{1/\theta}\Big\|_p &\lesssim&
\Big\|\Big(\sum\limits_{j\in \N_{0}^d} 2^{s|j|_1\theta}\sum_{{u}\in \{0,1\}^d}
\mathcal{R}^{e(j+u)}_m(f,2^{-j+u},\cdot)^{\theta}\Big)^{1/\theta}\Big\|_p\\
   &\lesssim&     \Big\|\Big(\sum\limits_{j\in \N_{0}^d} 2^{s|j|_1\theta}
\mathcal{R}^{e(j)}_m(f,2^{-j},\cdot)^{\theta}\Big)^{1/\theta}\Big\|_p\,.
\eeqq
In view of Proposition \ref{local} we finish the proof.\\
\end{proof}


\section{Change of variables}\label{sec:proof_modi}

Let $\varphi \in C^{k}_{0}(\R)$ such that $\supp(\varphi) \subset [0,1]$, $\int_0^1 \varphi(t)\,dt = 1$, $\varphi(t)>0$
on $(0,1)$ and $\varphi^{(k)}$ has only finitely many zeros in $[0,1]$. Let further 
\begin{equation*}\label{def:psi}
    \psi(t) := \int_{-\infty}^t \varphi(\xi)\,d\xi,\quad t\in \R\,,
\end{equation*}
then $\psi'(t)=\varphi(t)$, $t\in \R$. We  study the boundedness of the ``change of variable'' operator 
$$
    T^{\psi}_d\colon\, f(x) \;\mapsto\; \Big(\prod\limits_{i=1}^d \psi'(x_i)\Big)\cdot 
			f\bigl(\psi(x_1),...,\psi(x_d)\bigr),\qquad 
			f\in L_1(\R^d)\,, \ x\in \R^d\,,
$$
in the situations
$$
    T^{\psi}_d:\Bspt(\R^d) \to \Bspt(\R^d) \quad \mbox{ and } \quad T^{\psi}_d:\Fspt(\R^d) \to \Fspt(\R^d)\,.
$$
The following theorems represent the main results of the paper.
\medskip

\begin{thm}\label{changemixed-B} Let $1\leq p\leq \infty$, $0<\theta \leq \infty$, $s>0$  and $\varphi\in C^k_0(\R)$
as above with $k> \lfloor s \rfloor+2$ if $p>1$ and $k> \lfloor s \rfloor+3$ if $p=1$. Then 
$$
    \|T^{\psi}_d f\|_{\Bspt} \;\lesssim_{\psi}\; \|f\|_{\Bspt},\quad f\in \Bspt(\R^d)\,.
$$
\end{thm}

\smallskip

\begin{thm}\label{changemixed-F} Let $1<p<\infty$, $1< \theta\le\infty$, $s>0$  and $\varphi\in C^k_0(\R)$ as
above with $k> \lfloor s \rfloor+2$. Then 
$$
    \|T^{\psi}_d f\|_{\Fspt} \;\lesssim_{\psi}\; \|f\|_{\Fspt},\quad f\in \Fspt(\R^d)\,.
$$
\end{thm}

\medskip

\begin{rem}\label{s>1mixed} 
\begin{enumerate}[(i)]
	\item Observe that the smoothness $k$ of kernel $\varphi$ in the $\mathbf F$-case does not have to grow
to infinity when $p$ tends
				to $1$. This result has to be compared with the mentioned result of Templyakov, see \eqref{expl} and 
				\cite[page 237]{Te03}.
	\item With the technique used in this paper we were not able to give a result in the case $p < 1$ for both,
				$\mathbf{B}$ and $\mathbf{F}$-spaces, not even in the univariate situation, see Lemma \ref{univ-f} below. Note,
				that also the case $p=1$ is open in the $\mathbf{F}$-case. However, we strongly conjecture that we can prove
				boundedness of the change of variable mapping if $1/2<p \leq 1$ in both cases. 
\end{enumerate}
\end{rem}

\noindent For the convenience of the reader we will first comment on a corresponding 
univariate result. 
\begin{lem}\label{univ-f} Let $1<p<\infty$, $1< \theta\le\infty$, $s>0$  and $\varphi\in C^k_0(\R)$ as above with $k>
\lfloor s \rfloor+2$. Then 
$$
    \|T^{\psi}_1 f\|_{F^s_{p,\theta}} \;\lesssim_{\psi}\; \|f\|_{F^s_{p,\theta}},\quad f\in F^s_{p,\theta}(\R)\,.
$$
\end{lem}


\bigskip

\noindent
Before we come to the proof of the univariate boundedness results we need some preparation. 
It will turn out that difference characterizations of the function spaces under consideration is a suitable approach to
analyze the respective operators.

\subsection{Proof of Lemma \ref{univ-f}}


Let us first prove a technical lemma on the boundedness of quotients of derivatives of $\varphi$. 
Such terms naturally appear if one bounds the $L_p$-norm of $T^\psi_1 f$. For the particular choice $\phi=\psi_k'$, with
$\psi_k$ from \eqref{psin}, the lemma below was proved in in~\cite[p.~238]{Te93}.

\begin{lem}\label{lem:Linfty} Let $1<p\leq \infty$ and $m\in\N_0$. Further let $\phi \in C_0^{k}(\R)$ with 
$\supp(\varphi) \subset [0,1]$,  $\varphi>0$ on $(0,1)$ and $k>\frac{mp}{p-1}+1$. Let us further assume that
the $k$th derivative $\varphi^{(k)}$ has only finitely many zeros in $[0,1]$. Then we have
$$
    \frac{|\varphi^{(n)}(t)|}{|\varphi(t)|^{1/p}} \in L_{\infty}([0,1])\,
$$
for all $0\leq n\leq m$. 
\end{lem}

\bproof It is enough to prove the result for $n=m$. If $p=\infty$ the result is obvious. Hence, assume $1<p<\infty$. We
put $\ell=k-1$. Using Taylor's theorem and the fact that $\phi^{(i)}(0)=0$ for all $i=0,...,k$, 
we obtain 
$$\phi(t)=\frac{1}{\ell!}\int_0^t \phi^{(\ell+1)}(\xi)\, (t-\xi)^{\ell} d \xi$$ as well as 
$$\phi^{(m)}(t)=\frac{1}{(\ell-m)!}\int_0^t \phi^{(\ell+1)}(\xi)\, (t-\xi)^{\ell-m} d \xi$$
for all $t\in[0,1]$. Since $\varphi^{(\ell+1)}$ has only finitely many zeros in $[0,1]$ there exists an $\eps>0$ such
that $\phi^{(\ell+1)}(t)>0$ for $t\in(0,\eps)$. 
This shows with $p'=p/(p-1)$ that
\beqq
\frac{|\phi^{(m)}(t)|}{|\phi(t)|^{1/p}} 
\,&\le &\, \frac{\ell!}{(\ell-m)!}\, \frac{t^{1/p'} 
	\Big(\int_0^t \big(\phi^{(\ell+1)}(\xi)\big)^{p}\, (t-\xi)^{p(\ell-m)} d \xi\Big)^{1/p}}{
	 \Big(\int_0^t \phi^{(\ell+1)}(\xi)\, (t-\xi)^{\ell} d \xi\Big) ^{1/p}}\\
&\lesssim_k &\, \sup_{\xi\in(0,\eps)} \Big( \abss{\phi^{(\ell+1)}(\xi)}^{p-1}\, |t-\xi|^{p(\ell-m)-\ell}
\Big)^{1/p}\frac{  
	\Big(\int_0^t  \phi^{(\ell+1)}(\xi) \, (t-\xi)^{\ell} d \xi\Big)^{1/p}}{
	\Big(\int_0^t \phi^{(\ell+1)}(\xi)\, (t-\xi)^{\ell} d \xi\Big)^{1/p}} \\
&\lesssim_k & 1
\eeqq
for $t\in(0,\eps)$. In the last inequality we used $p>1$ and $\ell>\frac{pm}{p-1}$. The same arguments work also for
$(1-\eps_0,1)$ with some $\eps_0>0$. The quotient is uniformly bounded in $[\eps,1-\eps_0]$
since $\phi(t)\ge c>0$ for $t\in[\eps,1-\eps_0]$.\\
\eproof

It is easily seen that Lemma \ref{lem:Linfty} is not true for $p=1$. We immediately obtain the following corollary. 
Note the relaxed conditions on $p$ and $k$.

\begin{cor}\label{cor:Linfty}
Let $1\le p\leq \infty$ and $m\in\N_0$. Further let $\phi \in C_0^{k}(\R)$ with 
$\supp(\varphi) \subset [0,1]$,  $\varphi>0$ on $(0,1)$ and $k>m+1$. Let us further assume that
the $k$th derivative $\varphi^{(k)}$ has only finitely many zeros in $[0,1]$. Then we have
$$
    \frac{|\varphi^{(r)}(t)\varphi^{(\alpha)}(t)|}{|\varphi(t)|^{1/p}} \in L_{\infty}([0,1])\,
$$
for all $0\leq r,\alpha\leq m$ with $r+\alpha\le m$. 
\end{cor}

\begin{proof}
Obviously, it is enough to prove the statement for $p=1$, since
\[
\frac{|\varphi^{(r)}(t)\varphi^{(\alpha)}(t)|}{|\varphi(t)|^{1/p}}
\;\lesssim\; \frac{|\varphi^{(r)}(t)\varphi^{(\alpha)}(t)|}{|\varphi(t)|}
\]
due to the boundedness of $\varphi$. 
We write
\[
\frac{|\varphi^{(r)}(t)\varphi^{(\alpha)}(t)|}{|\varphi(t)|}
\;=\; \frac{|\varphi^{(r)}(t)|}{|\varphi(t)|^{1/p_1}} \cdot \frac{|\varphi^{(\alpha)}(t)|}{|\varphi(t)|^{1/p_2}}
\]
for some $p_1$, $p_2$ with $1/p_1+1/p_2=1$. 
Using Lemma~\ref{lem:Linfty} we know that we need 
$k>\frac{r p_1}{p_1-1} + 1$ for the boundedness of the first factor and 
$k>\frac{\alpha p_2}{p_2-1} + 1$ for the boundedness of the second.

First note that the statement of the corollary is trivial for $r=\alpha=0$.
If only one of them is zero, say $r=0$, then we set $p_1=1$ and $p_2=\infty$; 
again the statement is obvious.
Hence, we assume now that $r\neq0$ and $\alpha\neq0$ and we 
choose $p_1=(r+\alpha)/\alpha$ and $p_2=(r+\alpha)/r$. 
 
This leads to the restriction $k>r+\alpha+1$ in both cases. 
Since we want a result for all $r+\alpha\le m$ we impose the restriction $k>m+1$. 
This finishes the proof.\\
\end{proof}

Now we are in position to prove Lemma~\ref{univ-f}.\\
\\
\noindent
{\bf Proof of Lemma \ref{univ-f}.} {\it Step 1.} Let $\eta_j$ denote a smooth dyadic compactly supported decomposition
of unity (for instance the one given at the beginning of Subsection~\ref{subsec:spaces_R}). 
In particular, $1 = \sum_{j\in \N_0} \eta_j(t)$ for all $t\in\R$\, and $\eta_j(t) = 0$ if $|t|>2^{j+1}$. We decompose
$f$ as follows 
\be\label{decom}
    f(t) \,=\, \sum\limits_{j\in \Z} f_j(t) \,:=\, \sum\limits_{j\in \Z}\F^{-1}[\eta_j \F f](t)\,,
\ee
where we put $f_j \equiv 0$ in case $j<0$\,. 
Note, that with this definition, we can clearly write also $f=\sum_{\ell\in\Z} f_{j+\ell}$ 
for every fixed $j\in\N_0$. 
Using the characterization by rectangular means of differences of
$F_{p,\theta}^s$, see Theorem \ref{fdiff}, and the decomposition \eqref{decom} we have 
\begin{equation}
\begin{split}\label{f-com}
    \|T^{\psi}_1f\|^{(m)}_{F^s_{p,\theta}}
    \;&\leq\;   \sum\limits_{\ell\in \Z} \Big\|\Big(\sum\limits_{j\in
\N_0}2^{js\theta}\Big[2^j\int_{-2^{-j}}^{2^{-j}}|\Delta_h^m(T^{\psi}_1f_{j+\ell},\cdot)|\,dh\Big]^{\theta}\Big)^{
1/\theta}\Big\|_p\\
     &=\; \sum\limits_{\ell<0} \Big\|\Big(\sum\limits_{j\in
\N_0}2^{js\theta}\Big[2^j\int_{-2^{-j}}^{2^{-j}}|\Delta_h^{m}(T^{\psi}_1f_{j+\ell},\cdot)|\,dh\Big]^{\theta}\Big)^{
1/\theta}\Big\|_p \\
    &\qquad +\; \sum\limits_{\ell\geq 0} \Big\|\Big(\sum\limits_{j\in
\N_0}2^{js\theta}\Big[2^j\int_{-2^{-j}}^{2^{-j}}|\Delta_h^{m}(T^{\psi}_1f_{j+\ell},\cdot)|\,dh\Big]^{\theta}\Big)^{
1/\theta}\Big\|_p\,.
\end{split}
\end{equation}
Here we choose $m=\lfloor s \rfloor+1$.\\
{\it Step 2.} 
 First we recall the convolution representation of the $m$th order difference. Let $M_m(\cdot)$ is univariate B-spline
of degree $m$ which has knots at the points $\{0,1,...,m\}$, i.e.,
\beqq
M_m = \chi_{[0,1]}*\cdots *\chi_{[0,1]},\quad\text{ $m$ times},
\eeqq
where $\chi_{[0,1]}$ denotes the characteristic function of $[0,1]$. 
Note, that $M_m(\cdot)$ is bounded and $\supp(M_m)\subset [0,m]$. In case $h>0$ we have
\beqq
\Delta_h^{m}(T_1^{\psi}f_{j+\ell},t)=h^{m-1}\int_{0}^{hm}
\big(T_1^{\psi}f_{j+\ell}\big)^{(m)}(t+\xi)M_m(h^{-1}\xi)d\xi\,,
\eeqq
see \cite[page 45]{Devo93}. If $h<0$ we use B-spline with knots $\{-m,...,0\}$ and get the similar formula. Consequently, we obtain for $|h|<2^{-j}$
\be 
|\Delta_h^{m}(T_1^{\psi}f_{j+\ell},t)|\ \lesssim \ |h|^{m} M\big[(T_1^{\psi}f_{j+\ell}\big)^{(m)}\big](t)\ \leq\ 2^{-jm}M\big[(T_1^{\psi}f_{j+\ell}\big)^{(m)}\big](t)\,,\label{bm}
\ee
since $M_m(h^{-1}\cdot)$ is bounded. This leads to
\beqq
    &&\sum\limits_{\ell < 0} \Big\|\Big(\sum\limits_{j\in
\N_0}2^{js\theta}\Big[2^j\int_{-2^{-j}}^{2^{-j}}|\Delta_h^{m}(T^{\psi}_1f_{j+\ell},\cdot)|\,dh\Big]^{\theta}\Big)^{
1/\theta}\Big\|_p\\
    &&\qquad\lesssim~~ \sum\limits_{\ell < 0} \Big\|\Big(\sum\limits_{j\in
    \N_0}2^{(s-m)j\theta} \big(M\big[(T_1^{\psi}f_{j+\ell}\big)^{(m)}\big](\cdot)\big)^{\theta}\Big)^{
    1/\theta}\Big\|_p\\
    &&\qquad\lesssim ~~\sum\limits_{\ell < 0} \Big\|\Big(\sum\limits_{j\in
\N_0}2^{(s-m)j\theta}\big|\big(T_1^{\psi}f_{j+\ell}\big)^{(m)}(\cdot)\big|^{\theta}\Big)^{1/\theta}\Big\|_p\,.
\eeqq
Where in the second inequality we have applied Theorem \ref{feffstein}. Now we continue with
\beqq
\big(T_1^{\psi}f_{j+\ell}\big)^{(m)}(t)=\big[  \varphi f_{j+\ell}(\psi)\big]^{(m)}(t) = \sum_{\alpha=0}^m
C_{\alpha}\varphi^{(\alpha)}(t) [f_{j+\ell}(\psi)]^{(m-\alpha)}(t)\,,
\eeqq
for some $C_{\alpha}$. We distinguish two cases. First let us deal with $\alpha=m$. We choose the number $a$ such that
$\max\{1/p,1/\theta\}<a\leq m$. This is possible since $\max\{1/p,1/\theta\}<1$ and $m\geq 1$. Then, 
if $t\in [0,1]$ we have 
\begin{equation}\begin{split} \label{ct4}
\big|\varphi^{(m)}(t)f_{j+\ell}(\psi(t))\big|\ &\lesssim\ \sup\limits_{\xi\in [0,1]}|f_{j+\ell}(\xi)| 
\;=\; \sup\limits_{\xi\in [0,1]} \frac{|f_{j+\ell}(\xi)|}{(1+2^{j+\ell}|t-\xi|)^a} (1+2^{j+\ell}|t-\xi|)^a\\
\;&\lesssim\; 2^{(\ell+j)a} \sup\limits_{\xi\in [0,1]} \frac{|f_{j+\ell}(\xi)|}{(1+2^{j+\ell}|t-\xi|)^a}\ \\
&\lesssim\ 2^{(\ell+j)m} \sup\limits_{\xi\in \R} \frac{|f_{j+\ell}(\xi)|}{(1+2^{j+\ell}|t-\xi|)^a} \\
&=2^{(\ell+j)m}P_{j+\ell,a}(f_{j+\ell},t)\,. 
\end{split}
\end{equation}
If $t\not \in [0,1]$ the inequality \eqref{ct4} is obviously satisfied since $\supp(\varphi) \subset [0,1]$. We consider
the summand
\beq
    \Big\|\Big(\sum\limits_{j\in
\N_0}2^{(s-m)j\theta}\big|\varphi^{(m)}(\cdot)f_{j+\ell}(\psi(\cdot))\big|^{\theta}\Big)^{1/\theta}\Big\|_p
     &\lesssim &2^{(m-s)\ell}\Big\|\Big(\sum\limits_{j\in \N_0}2^{(\ell+j)s\theta}P_{j+\ell,a}(f_{j+\ell},\cdot)^
{\theta}\Big)^{1/\theta}\Big\|_p\nonumber\\
      &\lesssim& 2^{(m-s)\ell}\|f\|_{F^s_{p,\theta}}\,.\label{firstsum}
\eeq
In the last step we used Theorem \ref{peetremax}.
In case $\alpha<m$ we have
\beqq
\varphi^{(\alpha)}(t) [f_{j+\ell}(\psi)]^{(m-\alpha)}(t)
&=&\sum_{1\leq \gamma\leq m-\alpha}f_{j+\ell}^{(\gamma)}(\xi)|_{\xi=\psi(t)}\cdot\varphi^{(\alpha)} (t)\cdot
Q_{\gamma}(t)\,,
\eeqq
where $Q_{\gamma}(t)$ is the sum of the forms $\big(\varphi^{(\beta_1)}(t)\big)^{n_1}\big(\varphi^{(\beta_2)}(t)\big)^{n_2}$ 
with some $\beta_i,n_i\in \N_0$, $i=1,\,2$, satisfied $\beta_1n_1+\beta_2n_2=\gamma$. 
Note that the highest order derivative of $\varphi(t)$ in $Q_{\gamma}(t)$ is $m-\alpha$. 
Since  $\varphi\in C_0^k(\R)$, $k>\lfloor s \rfloor+2=m+1$, Corollary~\ref{cor:Linfty} yields that
\be\label{ct0}
\frac{|\varphi^{(r)}(t)\varphi^{(\alpha)}(t)|}{|\varphi(t)|^{1/p}} < C,
\ee
for all $r, \alpha$ with $r+\alpha\le m$, which leads to $|\varphi^{(\alpha)} (t) Q_{\gamma}(t)|\lesssim |\varphi(t)|^{1/p}$. 
Consequently,
\beqq
\big|\varphi^{(\alpha)}(t) [f_{j+\ell}(\psi)]^{(m-\alpha)}(t)\big|
&\lesssim & \sum_{1\leq \gamma\leq m-\alpha}\big|f_{j+\ell}^{(\gamma)}(\xi)\big|_{\xi=\psi(t)}\cdot|\varphi 
(t)|^{1/p}\,.
\eeqq
We consider the term
\beq\label{f24}
F_{\alpha,\ell}&:=&\Big\|\Big(\sum\limits_{j\in \N_0}2^{(s-m)j\theta}\big|\varphi^{(\alpha)}
[f_{j+\ell}(\psi)]^{(m-\alpha)}\big|^{\theta}\Big)^{1/\theta}\Big\|_p\nonumber\\
&\lesssim& \sum_{1\leq \gamma\leq m-\alpha}\Big\|\Big(\sum\limits_{j\in
\N_0}2^{(s-m)j\theta}\big|f_{j+\ell}^{(\gamma)}(\xi)\big|_{\xi=\psi(\cdot)}\cdot\varphi^{1/p}
\big|^{\theta}\Big)^{1/\theta}\Big\|_p\\
&\lesssim& \sum_{1\leq \gamma\leq m-\alpha}\Big\|\Big(\sum\limits_{j\in
\N_0}2^{(s-m)j\theta}\big|f_{j+\ell}^{(\gamma)}\big|^{\theta}\Big)^{1/\theta}\Big\|_p\,.\nonumber
\eeq
By Lemma~\ref{maxunglem} and a homogeneity argument we have for any $a>0$
\be\label{f31}
        |f^{(\gamma)}_{j+\ell}(t)| \;\leq\; \sup\limits_{\xi\in \R}
\frac{|f^{(\gamma)}_{j+\ell}(\xi)|}{(1+2^{j+\ell}|t-\xi|)^a} \;\lesssim\; 2^{(j+\ell)m}\sup\limits_{\xi\in \R}
\frac{|f_{j+\ell}(\xi)|}{(1+2^{j+\ell}|t-\xi|)^a}
\ee
if $1\leq \gamma\leq m$\,. Putting this into \eqref{f24} with $a>\max\{1/p,1/\theta\}$ and using Theorem \ref{peetremax}
afterwards yields
\begin{equation}\label{secondsum}
 \begin{split}
 F_{\alpha,\ell}  &\lesssim \sum_{1\leq \gamma\leq m-\alpha}2^{\ell(m-s)}\Big\|\Big(\sum\limits_{j\in
\N_0}2^{(j+\ell)s\theta} \big|P_{j+\ell,a}(f_{j+\ell},\cdot)\big|^{\theta}\Big)^{1/\theta}\Big\|_p\\
  &\lesssim~ 2^{\ell(m-s)}\|f\|_{F^s_{p,\theta}}\,.
\end{split}
  \end{equation}
From \eqref{firstsum}, \eqref{f24}, and \eqref{secondsum} we obtain
\beq\label{f26}
  \sum\limits_{\ell < 0} \Big\|\Big(\sum\limits_{j\in
\N_0}2^{js\theta}\Big[2^j\int_{-2^{-j}}^{2^{-j}}|\Delta_h^m(T^{\psi}_1f_{j+\ell},\cdot)|\,dh\Big]^{\theta}\Big)^{
1/\theta}\Big\|_p &\lesssim& \sum\limits_{\ell<0}2^{\ell(m-s)}\|f\|_{F^s_{p,\theta}}\nonumber\\
  &\lesssim & \|f\|_{F^s_{p,\theta}}\,.
\eeq

\noindent
{\it Step 3.} It remains to deal with $\sum_{\ell\geq 0}$. We expand the difference and estimate
\beq
    &&\sum\limits_{\ell \ge 0} \Big\|\Big(\sum\limits_{j\in
\N_0}2^{js\theta}\Big[2^j\int_{-2^{-j}}^{2^{-j}}|\Delta_h^{m}(T^{\psi}_1f_{j+\ell},\cdot)|\,dh\Big]^{\theta}\Big)^{
1/\theta}\Big\|_p\nonumber\\
   &&\qquad \qquad \lesssim~ \sum\limits_{\ell \ge 0} \sum_{i=0}^{m}\Big\|\Big(\sum\limits_{j\in
\N_0}2^{js\theta}\Big[2^j\int_{-2^{-j}}^{2^{-j}}|T^{\psi}_1f_{j+\ell}(\cdot+ih)|\,dh\Big]^{\theta}\Big)^{1/\theta}
\Big\|_p\,.\label{f25}
\eeq
We consider the term
\beq
A_{i,\ell} &=&\Big\|\Big(\sum\limits_{j\in
\N_0}2^{js\theta}\Big[2^j\int_{-2^{-j}}^{2^{-j}}|(T^{\psi}_1f_{j+\ell})(\cdot+ih)|\,dh\Big]^{\theta}\Big)^{1/\theta}
\Big\|_p \nonumber\\
&\lesssim &\Big\|\Big(\sum\limits_{j\in
\N_0}2^{js\theta}\big(M[T^{\psi}_1f_{j+\ell}](\cdot)\big)^{\theta}\Big)^{1/\theta}
\Big\|_p \,.\label{f25-1}
\eeq
Here again we use the Hardy-Littlewood maximal function. From Theorem \ref{feffstein} again we get
\beq
A_{i,\ell} &\lesssim& \Big\|\Big(\sum\limits_{j\in
\N_0}2^{js\theta}\big|(T^{\psi}_1f_{j+\ell})\big|^{\theta}\Big)^{1/\theta}\Big\|_p \nonumber\\
     &\lesssim&  2^{-\ell s}\Big\|\Big(\sum\limits_{j\in
\N_0}2^{(j+\ell)s\theta}|\varphi(\cdot) f_{j+\ell}(\psi(\cdot))|^{\theta}\Big)^{1/\theta}\Big\|_p\nonumber\\
     &\lesssim & 2^{-\ell s}\Big\|\Big(\sum\limits_{j\in
\N_0}2^{(j+\ell)s\theta}|f_{j+\ell}|^{\theta}\Big)^{1/\theta}\Big\|_p\nonumber\\
     &\lesssim & 2^{-\ell s}\|f\|_{F^s_{p,\theta}}\,.\label{f27}
\eeq
We performed a change of variable in the second step and used $p> 1$ and 
$\phi \in C_0^k(\R)$. With \eqref{f25}, \eqref{f25-1}, and \eqref{f27} we obtain
\beq
     &&\sum\limits_{\ell \ge 0} \Big\|\Big(\sum\limits_{j\in
\N_0}2^{js\theta}\Big[2^j\int_{-2^{-j}}^{2^{-j}}|\Delta_h^{m}(T^{\psi}_1f_{j+\ell},\cdot)|\,dh\Big]^{\theta}\Big)^{
1/\theta}\Big\|_p\nonumber\\
&&\qquad\qquad\qquad\qquad\lesssim ~\sum\limits_{i=1}^m\sum\limits_{\ell \geq 0} A_{i,\ell}\lesssim
\sum\limits_{\ell\geq 0} 2^{-\ell
s}\|f\|_{F^s_{p,\theta}} \lesssim \|f\|_{F^s_{p,\theta}}\,,\label{f27-1}
\eeq
where we used $s>0$ in the last inequality. 
Altogether we obtain from \eqref{f26} and \eqref{f27-1} 
$$
    \Big\|\Big(\sum\limits_{j\in \N_0}
2^{js\theta}\Big[2^j\int_{-2^{-j}}^{2^{-j}}|\Delta_h^m(T^{\psi}_1f,\cdot)|\,dh\Big]^{\theta}\Big)^{1/\theta}\Big\|_p
    \ \lesssim \ \|f\|_{F^s_{p,\theta}}\,.
$$
In view of \eqref{f-com} we conclude that
\beqq
 \|T^{\psi}_1f\|^{(m)}_{F^s_{p,\theta}} \lesssim \|f\|_{F^s_{p,\theta}}\,,
\eeqq
which finishes the proof. \\
\qed
\subsection{Proof of Theorems \ref{changemixed-B} and \ref{changemixed-F}}


\noindent
{\bf Proof of Theorem \ref{changemixed-F}.}
{\em Step 1.} We aim at adapting the proof of Lemma \ref{univ-f} to the multivariate mixed situation by applying the
arguments direction-wise. The proper tool will be the characterization by rectangular means of differences, see Theorem
\ref{fdiff}. Let $\{\eta_{j}\}_{j\in \N_0^d}$ be a tensorized compactly supported smooth dyadic decomposition of unity
and 
\begin{equation}\label{bblocks}
      f_j:=\F^{-1}[\eta_j \F f]\quad,\quad j\in \Z^d\,,
\end{equation}
where we put $f_j\equiv 0$ if $j\notin \N_0^d$\,. Analogously, to \eqref{f-com} we estimate the quantity
\beq 
    &&\Big\|\Big(\sum\limits_{j\in
\N_0^d}2^{|j|_1s\theta}\mathcal{R}^{e(j)}_m(T^{\psi}_df,2^{-j},\cdot)^{\theta}\Big)^{1/\theta}\Big\|_p\nonumber\\
    &&\qquad\qquad\qquad\leq~ \sum\limits_{\ell\in \Z^d} \Big\|\Big(\sum\limits_{j\in
\N_0^d}2^{|j|_1s\theta}\mathcal{R}^{e(j)}_m(T^{\psi}_df_{j+\ell},2^{-j},\cdot)^{\theta}\Big)^{1/\theta}\Big\|_p\,,\label{start2}
\eeq
where $2^{-j} = (2^{-j_1},...,2^{-j_d})$. Again we choose $m=\lfloor s \rfloor+1$. Let us introduce some further
notation first. If
$e\subset[d]$ and $j\in \N_0^d$ we denote

\begin{equation}\label{Gammaj}
\Gamma_j^{e}=  \bigtimes_{i\in e} [-2^{-j_i},2^{-j_i}]\,.
\end{equation}
For each $\ell\in \Z^d$ we denote 
\begin{equation}\label{ejl}
e_1(j,\ell)=e(j)\cap \{i: \ell_i\geq 0\}\qquad
\text{and}\qquad e_2(j,\ell)= e(j)\cap \{i: \ell_i<0\}. 
\end{equation}
For simplicity we put
\beqq
F_{\ell}:=\Big\|\Big(\sum\limits_{j\in
\N_0^d}2^{|j|_1s\theta}\mathcal{R}^{e(j)}_m(T^{\psi}_df_{j+\ell},2^{-j},\cdot)^{\theta}\Big)^{1/\theta}\Big\|_p\,.
\eeqq
{\it Step 2.} Estimation of $F_{\ell}$. We have
\beq
 \Delta_{h}^{m,e(j)} (T^{\psi}_df_{j+\ell},x)&=&\Delta_{h}^{m,e_1(j,\ell)}  \circ \Delta_{h}^{m,e_2(j,\ell)}
(T^{\psi}_df_{j+\ell},x)\nonumber \\
 &=&\sum_{k} \Delta_{h}^{m,e_2(j,\ell)}(T^{\psi}_df_{j+\ell},x+(k\cdot h))\label{delta}
\eeq
where the sum on the right-hand side is taken over all $k\in \N_0^d$ such that $0\leq k_i\leq m$ if $i\in e_1(j,\ell)$
and $k_i=0$ otherwise. In addition, we put $(k\cdot h)=(k_1h_1,...,k_dh_d)$. Using the convolution representation of the $m$th order difference in $e_2(j,\ell)$ we have the inequality
\beqq
&&\big| \Delta_{h}^{m,e_2(j,\ell)}\big(T^{\psi}_df_{j+\ell},x+(k\cdot h)\big)\big|\\
&&\qquad\qquad\lesssim\, \Big(\prod_{i\in
e_2(j,\ell)}2^{-j_im}\Big)M_{e_2(j,\ell)}\big[D^{(m_{e_2(j,\ell)})}(T^{\psi}_df_{j+\ell})\big](x+(k\cdot h))\,,
\eeqq
see \eqref{bm}. Here $m_{e_2(j,\ell)}=(m_1,...,m_d)$ where $m_i=m$ if $i\in e_2(j,\ell)$ otherwise $m_i=0$. This leads to
\beqq
 &&\mathcal{R}^{e(j)}_m(T^{\psi}_df_{j+\ell},2^{-j},x) \\
 &&\qquad\lesssim~ \sum_{k}  \int_{\Gamma_j^{e(j)}} 
     \big|\Delta_{h}^{m,e_2(j,\ell)}\big(T^{\psi}_df_{j+\ell},x+(k\cdot h)\big)\big|\prod_{i\in e(j)}2^{j_i}dh_i\\
    &&\qquad\lesssim ~  \sum_{k}  \int_{\Gamma_j^{e(j)}} \Big(\prod_{i\in
    e_2(j,\ell)}2^{-j_im}\Big)M_{e_2(j,\ell)}\big[D^{(m_{e_2(j,\ell)})}(T^{\psi}_df_{j+\ell})\big](x+(k\cdot h))\prod_{i\in e(j)}2^{j_i}dh_i\,.
\eeqq 
Now we estimate above by the Hardy-Littlewood maximal function with the components in $e_1(j,\ell)$ and obtain
\beqq
 \mathcal{R}^{e(j)}_m(T^{\psi}_df_{j+\ell},2^{-j},x) &\lesssim &  \Big(\prod_{i\in e_2(j,\ell)}
2^{-j_im}\Big)M_{e_1(j,\ell)}\big[M_{e_2(j,\ell)}\big[D^{(m_{e_2(j,\ell)})}(T^{\psi}_df_{j+\ell})\big]\big](x)\\
&\lesssim &  \Big(\prod_{i\in e_2(j,\ell)}
2^{-j_im}\Big)M_{[d]}\big[D^{(m_{e_2(j,\ell)})}(T^{\psi}_df_{j+\ell})\big](x) \,.
\eeqq
Plugging this into $F_{\ell}$, Theorem \ref{feffstein} yields
\beqq
F_{\ell}
 &\lesssim &\Big\|\Big(\sum\limits_{j\in \N_0^d}2^{|j|_1s\theta}\Big(\prod_{i\in
e_2(j,\ell)}2^{-j_im}\Big)\big|D^{(m_{e_2(j,\ell)})}(T^{\psi}_df_{j+\ell})(\cdot)\big|^{\theta}\Big)^{1/\theta}\Big\|_p\,.
\eeqq
Using Leibniz's formula we estimate
\begin{equation}\begin{split} \label{ine1}
\big|D^{(m_{e_2(j,\ell)})}(T^{\psi}_df_{j+\ell})(x)\big|& = \Big|D^{(m_{e_2(j,\ell)})}\Big( \big[
\bigotimes_{i=1}^d\varphi\big] f_{j+\ell}(\psi ,...,\psi )\Big)(x)\Big|\\
&\lesssim \sum_{\alpha}\Big| D^{(\alpha)} \Big[ \bigotimes_{i=1}^d\varphi\Big](x) D^{(\beta)}[f_{j+\ell}(\psi ,...,\psi
)](x)\Big|\,. 
\end{split}\end{equation}
Here the sum on the right-hand side is taken over all $\alpha\in \N_0^d$ with $0\leq\alpha_i\leq m$ if $i\in
e_2(j,\ell)$, otherwise $\alpha_i=0$ and $\beta=m_{e_2(j,\ell)}-\alpha.$ We continue estimation by dividing
$e_2(j,\ell)=e_2^1(j,\ell)\cup e_2^2(j,\ell)$ where $\alpha_i<m$ if $i\in e_2^1(j,\ell)$ and $\alpha_i=m$ if $i\in
e_2^2(j,\ell)$. We have
\beqq
&&D^{(\alpha)} \Big[ \bigotimes_{i=1}^d\varphi\Big](x) D^{(\beta)}[f_{j+\ell}(\psi ,...,\psi )](x)\\
&&\quad=~\Big(\prod_{i\in e_0(j)\cup e_1(j,\ell)}\varphi (x_i)\Big)\Big(\prod_{i\in e_2^2(j,\ell)}
\varphi^{(m)}(x_i)\Big)\Big(\prod_{i\in e_2^1(j,\ell)} \varphi^{(\alpha_i)}(x_i)\Big)D^{(\beta)}[f_{j+\ell}(\psi
,...,\psi )](x)\,.
\eeqq
Now the condition $\varphi \in C^k_0(\R)$, 
$k>m+1$, 
implies that 
\beqq
&&\Big|\Big(\prod_{i\in e_2^1(j,\ell)} \varphi^{(\alpha_i)}(x_i)\Big)D^{(\beta)}[f_{j+\ell}(\psi ,...,\psi )](x)\Big| \\
&&\qquad\qquad\qquad\qquad\qquad\lesssim~  \Big(\prod_{i\in e_2^1(j,\ell)}|\varphi (x_i)|^{1/p}
\Big)\cdot\big|[D^{(\beta)}f_{j+\ell}](\psi(x_1),...,\psi(x_d)\big|,
\eeqq
see \eqref{ct0}, which leads to
\beqq
&&\Big|D^{(\alpha)} \Big[ \bigotimes_{i=1}^d\varphi\Big](x) D^{(\beta)}[f_{j+\ell}(\psi ,...,\psi )](x)\Big|\\
&&\qquad\qquad \lesssim \Big(\prod_{i\in [d]\backslash e_2^2(j,\ell)}|\varphi (x_i)|^{1/p}
\Big)\cdot\Big|\Big(\prod_{i\in e_2^2(j,\ell)}
\varphi^{(m)}(x_i)\Big)[D^{(\beta)}f_{j+\ell}](\psi(x_1),...,\psi(x_d)\Big|\,.
\eeqq
Changing variable we obtain
\begin{equation}\begin{split}\label{ine2}
&\Big\|\Big(\sum\limits_{j\in \N_0^d}2^{|j|_1s\theta}\Big(\prod_{i\in e_2(j,\ell)}2^{-j_im\theta}\Big)\Big|D^{(\alpha)}
\Big[ \bigotimes_{i=1}^d\varphi\Big]\cdot D^{(\beta)}[f_{j+\ell}(\psi ,...,\psi)]
\Big|^{\theta}\Big)^{1/\theta}\Big\|_p\\
&\lesssim\;\Big\|\Big(\sum\limits_{j\in \N_0^d}2^{|j|_1s\theta}\Big(\prod_{i\in
e_2(j,\ell)}2^{-j_im\theta}\Big)\Big|\Big(\prod_{i\in e_2^2(j,\ell)}
\varphi^{(m)}(x_i)\Big)[D^{(\beta)}f_{j+\ell}](x,\psi_{e_2^2(j,\ell)} 
)\Big|^{\theta}\Big)^{1/\theta}\Big\|_p\,.
\end{split}
\end{equation}
Here $[D^{(\beta)}f_{j+\ell}](x,\psi_{e_2^2(j,\ell)})=[D^{(\beta)}f_{j+\ell}](z_1,...,z_d)$ with $z_i=\psi(x_i)$ if
$i\in e_2^2(j,\ell) $ otherwise $z_i=x_i$. Using the same scaling argument as in \eqref{ct4} and \eqref{f31} yields
\begin{equation}\begin{split}\label{ine3}
\Big|\Big(\prod_{i\in e_2^2(j,\ell)}& \varphi^{(m)}(x_i)\Big)D^{(\beta)}[f_{j+\ell}](x,\psi_{e_2^2(j,\ell)} )\Big|\\
&\quad\lesssim~ \Big(\prod\limits_{i\in e_2^2(j,\ell)}2^{(j_i+\ell_i)m}\Big)\sup\limits_{y\in
\R^{d}}\frac{|D^{\beta}f_{j+\ell}(y)|}{\prod\limits_{i=1}^d(1+2^{j_i+\ell_i}|x_i-y_i|)^a}\\
&\quad\lesssim~  \Big(\prod\limits_{i\in e_2^2(j,\ell)}2^{(j_i+\ell_i)m}\Big)\Big(\prod\limits_{i\in
e_2^1(j,\ell)}2^{(j_i+\ell_i)\beta_i}\Big)\sup\limits_{y\in \R^{d}}\frac{|
f_{j+\ell}(y)|}{\prod\limits_{i=1}^d(1+2^{j_i+\ell_i}|x_i-y_i|)^a}\\
&\quad\lesssim~  \Big(\prod\limits_{i\in
e_2(j,\ell)}2^{(j_i+\ell_i)m}\Big)P_{2^{j+\ell},a}f_{j+\ell}(x)\,
\end{split}\end{equation}
with $m\geq a$. Now we obtain from \eqref{ine1}, \eqref{ine2} and \eqref{ine3}
\beqq
F_{\ell}
 &\lesssim &\Big\|\Big(\sum\limits_{j\in \N_0^d}2^{|j|_1s\theta} \Big(\prod\limits_{i\in e_2(j,\ell)}2^{\ell_i
m\theta}\Big)\big|P_{2^{j+\ell},a}f_{j+\ell}\big|^{\theta}\Big)^{1/\theta}\Big\|_p\\
 &\lesssim &\Big\|\Big[\sum\limits_{j\in \N_0^d}\Big( \prod\limits_{i\in e_2(j,\ell)}2^{\ell_i
(m-s)}\Big)\Big(\prod\limits_{i\in e_0(j)\cup e_1(j,\ell)}2^{-s\ell_i } \Big)2^{|j+\ell|_1s\theta}
\big|P_{2^{j+\ell},a}f_{j+\ell}\big|^{\theta}\Big]^{1/\theta}\Big\|_p\,\\
  &\lesssim &\Big\|\Big[\sum\limits_{j\in \N_0^d}\Big(\prod\limits_{i:\, \ell_i<0}2^{\ell_i
(m-s)}\Big)\Big(\prod\limits_{i:\, \ell_i\geq 0}2^{-s\ell_i }\Big)  2^{|j+\ell|_1s\theta}
\big|P_{2^{j+\ell},a}f_{j+\ell}\big|^{\theta}\Big]^{1/\theta}\Big\|_p\,
\eeqq
with $a>\max\{1/p,1/\theta\}$. This is because $i\in e_0(j)$ implies $\ell_i\geq 0$.  Then Theorem \ref{peetremax}
yields
\beqq
F_{\ell}
 &\lesssim & \Big(\prod\limits_{i:\, \ell_i<0}2^{\ell_i (m-s)}\Big)\Big(\prod\limits_{i:\, \ell_i\geq 0}2^{-s\ell_i
}\Big)   \|f\|_{\Fspt}\,.
\eeqq
{\it Step 3.} In a view of \eqref{start2} we conclude that
\begin{equation}\nonumber
     \Big\|\Big(\sum\limits_{j\in
\N_0^d}2^{|j|_1s\theta}\mathcal{R}^{e(j)}_m(T^{\psi}_df,2^{-j},\cdot)^{\theta}\Big)^{1/\theta}\Big\|_p
      ~\lesssim~ \|f\|_{\Fspt}\,.
\end{equation}
The proof is complete.\qed
\vskip 5mm
\noindent
{\bf Proof of Theorem \ref{changemixed-B}.} {\it Step 1.}  
Let $\{\eta_{j}\}_{j\in \N_0^d}$ and $\{f_{j}\}_{j\in\Z^d}$ be defined as in the proof of Theorem~\ref{changemixed-F}. 
Analogous to \eqref{start2} we have
\beqq
     \|T^{\psi}_df\|_{\mathbf{B}^s_{p,\theta}}^{(m)} 
& \leq&  \Big(\sum\limits_{j\in \N_0^d} 2^{|j|_1s\theta}\Big( \sum\limits_{\ell \in \Z^d}\big\|\mathcal{R}^{e(j)}_m(T^{\psi}_df,2^{-j},\cdot)\big\|_p\Big)^{\theta}\Big)^{1/\theta}\\
& \leq& \Big(\sum_{\ell \in \Z^d}\sum\limits_{j\in \N_0^d} 2^{|j|_1s\theta} \big\|\mathcal{R}^{e(j)}_m(T^{\psi}_df_{j+\ell},2^{-j},\cdot)\big\|_p^{\theta}\Big)^{1/\theta}\,.
\eeqq
Here we assume $\theta<1$. If $\theta\geq 1$ we use triangle inequality.\\
{\it Step 2.} The case $p>1$. We put $m=\lfloor s\rfloor +1$. The proof is similar in 
	$\mathbf F$-spaces but less technical. Here we use the classical Hardy-Littlewood maximal 
	inequality (the scalar version of Theorem \ref{feffstein}) and the scalar version of 
	Theorem \ref{peetremax}. These inequalities do not depend on the parameter $\theta$. Hence the result in the $\mathbf B$-spaces can be extended to $0<\theta\leq \infty$. \\
{\it Step 3.} The case $p=1$. Using the convolution representation of the $m$th order 
	difference in $e_2(j,\ell)$ we obtain from \eqref{delta}
\beqq
&& \Delta_{h}^{m,e(j)} (T^{\psi}_df_{j+\ell},x)\\
 &&\qquad=\sum_{k} \Delta_{h}^{m,e_2(j,\ell)}(T^{\psi}_df_{j+\ell},x+(k\cdot h))\\
 &&\qquad=\sum_{k}\int_{\R^{|e_2(j,\ell)|}}D^{(m_{e_2(j,\ell)})}(T^{\psi}_df_{j+\ell})(x+y+(k\cdot h))\prod_{i\in e_2(j,\ell)}h_i^{m-1}M_m(h_i^{-1}y_i)dy_i\,.
\eeqq
Here $y=(y_1,...,y_d)$ with $y_i=0$ if $i\not\in e_2(j,\ell)$. We used the same notation in 
the proof of Theorem \ref{changemixed-F}, see \eqref{Gammaj} and \eqref{ejl}. Inserting this into $\big\|\mathcal{R}^{e(j)}_m(T^{\psi}_df,2^{-j},\cdot)\big\|_1$ and then changing the order of integration with the fact that
$$ \prod_{i\in e_2(j,\ell)}\int_{\R}h_i^{-1} M_n(h_i^{-1}y_i) dy_i=1\,$$
 we obtain 
\beqq \big\|\mathcal{R}^{e(j)}_m(T^{\psi}_df,2^{-j},\cdot)\big\|_1 & \lesssim &    \Big(\prod_{i\in e_2(j,\ell)} 2^{-j_im}\Big) \int_{\Gamma_j^{e(j)}}\big \| D^{(m_{e_2(j,\ell)})}(T^{\psi}_df_{j+\ell})(\cdot )\big\|_1 \prod_{i\in e(j)}2^{j}dh_i \\
& \lesssim &  \Big(\prod_{i\in e_2(j,\ell)} 2^{-j_im}\Big) \big \| D^{(m_{e_2(j,\ell)})}(T^{\psi}_df_{j+\ell})(\cdot )\big\|_1\,.
\eeqq
Next step is carried out as $\mathbf F$-spaces. 
Note that in this case we choose $m=\lfloor s\rfloor +2$, since 
there must exist $a$ such that $1<a\leq m$ for the inequality \eqref{ine3} to hold. 
The proof is complete. \\
\qed

\begin{rem}\label{rem:p1}
The last step of the proof for $p=1$ shows that, based on our method, we have to guarantee 
that $\max\{1,s\}< m$. That's why we need the more restrictive condition 
$k> m+1=\lfloor s\rfloor+3$ in Theorem~\ref{t2}. 
Under the additional assumption $s\ge1$ we can relax this condition to 
$k> \lfloor s\rfloor+2$ as in the case $p>1$.

Another possibility to relax this condition is to use Nikol'skij inequality at the appropriate point, 
as it was already done in \cite{Du2}.
\end{rem}

\section{Pointwise multiplication}
\label{sec:mult}

We will first comment on the result for ${\mathbf F}$-spaces in Theorem \ref{pointmult-F}. 
Again, we make use of the useful characterization by differences in Theorem \ref{fdiff}. 
The technique used in the proof might not be ``optimal'' in the sense of Remarks
\ref{rem1}, \ref{rem2} below. However, it is quite transparent and works well for spaces with dominating mixed
smoothness ${\mathbf A}^s_{p,\theta}$ where almost nothing is known in this direction. Recall the multiplication
operator $\wt T_d^{\psi}$ is defined by
\[
\wt T_d^{\psi}\colon\, f(x) \;\mapsto\; \psi(x)\cdot f(x),\qquad 
			f\in L_1(\R^d)\,, \ x\in \R^d\,,
\]
for some sufficiently smooth and compactly supported function $\psi$. 
Similar to Lemma \ref{univ-f} we will proof a one-dimensional version 
of Theorem \ref{pointmult-F} first.

\begin{lem}\label{univ2} Let $0<p < \infty$, $0<\theta \leq\infty$ and $s>\sigma_{p,\theta}$. Let further
$\psi\in C_0^k(\Omega)$ with $k\geq \lfloor s \rfloor+1$ for some compact set $\Omega\subset\R$. 
Then 
\begin{equation}\label{mult_rel}
      \|{\wt T}^{\psi}_1 f\|_{F^s_{p,\theta}} \lesssim
      \|f\|_{F^s_{p,\theta}(\mathbb{T})},\quad f\in
F^s_{p,\theta}(\mathbb{T})\,.
\end{equation}
\end{lem}
\begin{rem}[Small values of $s$]\hspace{\fill}\label{rem1}\\
\indent {\em (i)} We believe that the conditions $s>\sigma_{p,\theta}$ and $s>\sigma_p$ in Lemma \ref{univ2} and
Theorems \ref{pointmult-B}, \ref{pointmult-F} are only technical and caused by our proof technique. On the one hand
$s>\sigma_{p,\theta}$ (and $s>\sigma_p$ in the $B$-case) is necessary for the characterization of Triebel-Lizorkin
spaces by means of differences, see Theorems \ref{diff}, \ref{fdiff} and \cite{ChSe06}, and on the other hand our proof
below requires this condition. One may relax (or even remove) this condition by using the characterization of
Triebel-Lizorkin spaces via
local means, see Proposition \ref{local} above, \cite[Thm.\ 4.2.2]{Tr92} and \cite[Thm.\ 1.25]{Vyb06}. Note that
removing the condition $s>\sigma_{p,\theta}$ will extend the range of parameters for which \eqref{per_circle}
holds. In fact, if $1/\theta-1> 1/p$ then the case $1/p\leq s<1/\theta-1$ for the ${\mathbf F}$-scale in
\eqref{per_circle} is not covered by Theorem \ref{pointmult-F}, although it is reasonable since ${\mathbf
F}^s_{p,\theta} \hookrightarrow C(\R^d)$. In other words, we would obtain

$$
   \mbox{Int}_n(\mathring{\mathbf A}_{p,\theta}^s) \asymp
   \mbox{Int}_n({\mathbf A}_{p,\theta}^s(\mathbb{T}^d))\,,\quad n\in \N\,,
$$
holds true {\it whenever} ${\mathbf A}_{p,\theta}^s$ is embedded into $C(\R^d)$ which represents a certain minimal
condition such that \eqref{f01} makes sense. 

{\em (ii)} The condition $s>\sigma_p$ ensures that the spaces ${\mathbf A}_{p,\theta}^s$ consist of regular
distributions represented by a locally integrable function, see Lemma \ref{emb},(i) above. Therefore, the pointwise
multiplication is well-defined. However, inspecting the proofs below we only need to give sense to the pointwise
multiplication for smooth building blocks, see \eqref{bblocks} above, such that we may define
\begin{equation}\label{para1}
    {\wt T}^{\psi}_d f := \sum\limits_{j\in \N_0^d} \psi(x) f_j(x)\,,
\end{equation}
where convergence is considered in $\mathcal{S}'(\R^d)$\,. From that point of view one may even drop the condition
$s>\sigma_p$ at the price of $k=k(p,s)$ in both, $B$ and $F$-spaces. We leave the details to the interested reader.
\end{rem}

\begin{rem}[Pointwise multiplier spaces and optimality] \label{rem2} \quad\\
\indent {\em (i)} We ask for a sharp description of the space of pointwise multipliers for spaces
$M({\mathbf A}^s_{p,\theta}(\R^d))$. This space might be much larger than $C_0^k(\Omega)$, which we considered in
Theorem \ref{pointmult-F}. Questions of this type have some history, see \cite[Chapt.\ 4]{RuSi96} and the
references therein. In
that sense, the result in Lemma \ref{univ2} is far from being optimal. In fact, combining the technique from \cite[Thm.\
4.9.1]{RuSi96} with \cite[Thms.\ 1.26, 1.29]{Tr08} one can prove the following rather sharp result: If
$A^s_{p,\theta}(\R^d)$ (which is the isotropic version of \eqref{ringspace}) embeds into $C(\R^d)$ then $\psi$ is a
pointwise multiplier in the sense above if $\psi \in
{\mathring A}^s_{p,\theta}(\Omega)$ . What concerns the case
$s<n/p$ we refer to \cite{Si99}.\\
\indent {\em (ii)} The technique used in \cite[Chapt.\ 4]{RuSi96} is not yet developed for spaces with dominating mixed
smoothness ${\mathbf A}^s_{p,\theta}$. This seems to be a difficult task and is not straight-forward. In fact, it is
for instance not clear whether all $\psi\in \mathring{\mathbf A}_{p,\theta}^s(\Omega)$ are pointwise multipliers in
${\mathbf A}^s_{p,\theta} \hookrightarrow C(\R^d)$.

\end{rem}

\subsection{Proof of Lemma \ref{univ2}}
For the proof we need the following variant of \cite[Lemma~3.3.1]{Ul06}.
\begin{lem}\label{1dim}Let $a>0$, $b\geq 1$, $m\in \N$, $h\in \R\setminus\{0\}$,
$\psi \in C_0^{k}(\R)$ with $k\geq m$ and $f\in \S'(\R)$ such that $\supp(\F f) \subset
[-b,b]$\,. Then it holds for any $t\in \R$
\begin{equation}\label{help}
    |\Delta^m_h(\psi f,t)| \leq
C_{m,a,\psi}\max\{1,|bh|^a\}\min\{1,|bh|^m\}P_{b,a}f(t)\,.
\end{equation}
\end{lem}

\bproof Assume first that $b=1$. By the mean value theorem from calculus we get 
\beqq
    |\Delta^m_h(\psi f,t)| &\leq& |h|^m \max\{1,|mh|^a\}\sup\limits_{|\xi|\leq
mh}\frac{|(\psi f)^{(m)}(t-\xi)|}{(1+|\xi|)^a}\\
    &\leq& c_{m,a} |h|^m \max\{1,|h|^a\} \sup\limits_{\xi \in \R}\frac{\big|\sum_{j=0}^m
\binom{m}{j}\psi^{(j)}(t-\xi)f^{(m-j)}(t-\xi)\big|}{(1+|\xi|)^a}\,,
\eeqq
where we used Leibniz' rule in the second estimate. Let us define
\begin{equation}\label{cmp}
    c_{m,\psi}:= 2^m\max\limits_{j=0,...,m}\|\psi^{(j)}\|_{\infty}
\end{equation}
and continue with triangle inequality and \eqref{g1} to obtain
\begin{equation}\label{9.38}
 |\Delta^m_h(\psi f)|\leq c_{m,a}\cdot c_{m,\psi}\cdot |h|^m \max\{1,|h|^a\}\cdot \sup\limits_{\xi \in
\R}\frac{|f(t-\xi)|}{(1+|\xi|)^a}\quad,\quad x\in \R.
\end{equation}
What follows is a simple homogeneity argument to deal with $b>1$ and $\supp(\F f) \subset [-b,b]$. Putting
$\psi_{b}:=\psi(\cdot/b)$ and $f_{b}:=f(\cdot /b)$ we get 
$\Delta^m_h(\psi f) = \Delta^m_{bh}(\psi_{b}f_{b})(bt)$. Further, we have
$c_{m,\psi_{b}} \lesssim c_{m,\psi}$ in \eqref{cmp}. Due to $\supp(\F f_{b}) \subset [-1,1]$ we can apply \eqref{9.38}
to obtain for $t\in \R$
\begin{equation}\begin{split}\label{9.39} 
  |\Delta^m_h(\psi f,t)| = |\Delta^m_{bh}(\psi_{b}f_{b},bt)| &\leq 
  c_{m,a}\cdot c_{m,\psi}\cdot |bh|^m\max\{1,|bh|^a\}\cdot \sup\limits_{\xi \in
\R}\frac{|f_{b}(bt-\xi)|}{(1+|\xi|)^a} \\
  &=  c_{m,a}\cdot c_{m,\psi}\cdot |bh|^m\max\{1,|bh|^a\}\cdot \sup\limits_{\xi \in
\R}\frac{|f(t-\xi)|}{(1+|b\xi|)^a}\\
  &= C_{m,a,\psi} |bh|^m\max\{1,|bh|^a\}P_{b,a}f(t)\,.
\end{split}\end{equation}
On the other hand, by \eqref{expdiff}, we observe for $t\in \R$
\begin{equation}\begin{split}\label{9.40}
 |\Delta_h^{m}(\psi f,t)| &\leq  \sum\limits_{j=0}^m\Big|\binom{m}{j}\psi(t+jh)\Big|\cdot |f(t + jh)|\\
      &\leq c_{m,\psi}c_{m,a}\max\{1,|bh|^a\}\sup\limits_{|\xi|\leq mh}\frac{|f(t-\xi)|}{(1+|b\xi|)^a}\\
      &\leq  C_{m,a,\psi}\max\{1,|bh|^a\}P_{b,a}f(t)\,.
\end{split}\end{equation}
Now, \eqref{9.39} together with \eqref{9.40} imply \eqref{help}.\\
\eproof

\noindent
{\bf Proof of Lemma \ref{univ2}.} In the case $\min\{p,\theta\}\leq
1$ we choose $0<\lambda<\min\{p,\theta\}$ and $a>1/\min\{p,\theta\}$ such that $s-(1-\lambda)a>0$. It is
possible since $s>\sigma_{p,\theta}$. If $\min\{p,\theta\}>1$ we put $\lambda=1$. 
Let $\Omega'=\Omega+[-1,1]$ and note that 
$\supp({\wt T}^{\psi}_1f)\subset\supp(\psi)\subset\Omega$ by assumption.
Then, clearly with $m\geq \lfloor s \rfloor+1$ we have
\begin{equation}\begin{split}\label{9.41a}
  \|{\wt T}^{\psi}_1f\|_{F^s_{p,\theta}}^{(m)} &\lesssim \Big\|\Big(\sum\limits_{j\in \N_0}
2^{js\theta}\Big[2^j\int_{-2^{-j}}^{2^{-j}}|\Delta^m_h(\wt
T^{\psi}_1f,\cdot)|\,dh\Big]^{\theta}\Big)^{1/\theta}\Big\|_{L_p(\Omega')}\\
  &\leq   \Big\|\Big(\sum\limits_{j\in \N_0}2^{js\theta}\Big[\sum\limits_{\ell\in
\Z}2^j\int_{-2^{-j}}^{2^{-j}}|\Delta^m_h(\wt
T^{\psi}_1f_{j+\ell},\cdot)|\,dh\Big]^{\theta}\Big)^{1/\theta}\Big\|_{L_p(\Omega')} \\
    &\leq   \Big\|\Big(\sum\limits_{\ell<0}\Big(\sum\limits_{j\in
\N_0}2^{js\theta}\Big[2^j\int_{-2^{-j}}^{2^{-j}}|\Delta^m_h(\wt
T^{\psi}_1f_{j+\ell},\cdot)|\,dh\Big]^{\theta}\Big)^{u/\theta}\Big)^{1/u}\Big\|_{L_p(\Omega')} \\
    &+     \Big\|\Big(\sum\limits_{\ell\geq 0}\Big(\sum\limits_{j\in
\N_0}2^{js\theta}\Big[2^j\int_{-2^{-j}}^{2^{-j}}|\Delta^m_h(\wt
T^{\psi}_1f_{j+\ell},\cdot)|\,dh\Big]^{\theta}\Big)^{u/\theta}\Big)^{1/u}\Big\|_{L_p(\Omega')} 
\end{split}\end{equation}
where $u=\min\{1,\theta\}$. By \eqref{help} we obtain for $\ell<0$ that 
\beqq
  2^j\int_{-2^{-j}}^{2^{-j}}|\Delta^m_h(\wt T^{\psi}_1f_{j+\ell},\cdot)|\,dh  \lesssim 2^{\ell
m}P_{2^{\ell+j},a}f_{j+\ell}(t)
\eeqq
and hence
\beqq
  &&\Big\|\Big(\sum\limits_{\ell<0}\Big(\sum\limits_{j\in
\N_0}2^{js\theta}\Big[2^j\int_{-2^{-j}}^{2^{-j}}|\Delta^m_h(\wt
T^{\psi}_1f_{j+\ell},\cdot)|\,dh\Big]^{\theta}\Big)^{u/\theta}\Big)^{1/u}\Big\|_{L_p(\Omega')}   \\
       &&\qquad\qquad\qquad\leq~ \Big\|\Big( \sum\limits_{\ell < 0}2^{\ell(m-s)} \Big(\sum\limits_{j\in
\N_0}2^{(j+\ell)s\theta}|P_{2^{j+\ell},a}f_{j+\ell}|^{\theta}\Big)^{u/\theta}\Big)^{1/u}\Big\|_{L_p(\Omega')}
\\
&&\qquad\qquad\qquad\lesssim~  \Big\| \Big(\sum\limits_{j\in
\N_0}2^{js\theta}|P_{2^{j},a}f_j|^{\theta}\Big)^{1/\theta}\Big\|_{L_p(\mathbb{T})}\,,       
\eeqq
since the function inside the $L_p$-(quasi-)norm is $1$-periodic. Applying the
periodic version of Theorem \ref{peetremax} with $a>\max\{1/p,1/\theta\}$, see \cite[Thm.\ 3.4.1]{ST}, we obtain
\be
   \Big\|\Big(\sum\limits_{\ell<0}\Big(\sum\limits_{j\in \N_0}2^{js\theta}\Big[2^j\int_{-2^{-j}}^{2^{-j}}|\Delta^m_h(\wt
T^{\psi}_1f_{j+\ell},\cdot)|\,dh\Big]^{\theta}\Big)^{u/\theta}\Big)^{1/u}\Big\|_{L_p(\Omega')} \ \lesssim\
\|f\|_{F^s_{p,\theta}(\mathbb{T})}\,.\label{1d}
\ee
Let us deal with $\ell \geq 0$. We write
$$
    |\Delta_h^m(\psi f_{j+\ell},t)| = |\Delta_h^m(\psi
f_{j+\ell},t)|^{1-\lambda}\cdot |\Delta_h^m(\psi f_{j+\ell},t)|^{\lambda}.  
$$
By \eqref{help} we obtain
$$
    |\Delta_h^m(\psi f_{j+\ell},t)|^{1-\lambda} \lesssim 2^{\ell
a(1-\lambda)}|P_{2^{j+\ell},a}f_{j+\ell}(t)|^{1-\lambda}\,.
$$
This gives 
$$
  2^j\int_{-2^{-j}}^{2^{-j}}|\Delta^m_h(\psi f_{j+\ell},t)|^{\lambda}\,dh \leq
M[|f_{j+\ell}|^{\lambda}](t)\,
$$
and therefore
\beqq
  &&  \sum\limits_{\ell\geq 0}\Big(\sum\limits_{j\in \N_0}2^{js\theta}\Big[2^j\int_{-2^{-j}}^{2^{-j}}|\Delta^m_h(\wt
T^{\psi}_1f_{j+\ell},\cdot)|\,dh\Big]^{\theta}\Big)^{u/\theta} \\
 &&\qquad\qquad\lesssim~ \sum\limits_{\ell\geq0}\Big(\sum\limits_{j\in \N_0}2^{js\theta}2^{\ell
  a(1-\lambda)\theta}\Big[|P_{2^{j+\ell},a}f_{j+\ell}(t)|^{1-\lambda}M[|f_{j+\ell}|^{\lambda}
  ](t)\Big]^{\theta}\Big)^{u/\theta}  \\ 
   &&\qquad\qquad \lesssim~ \sum\limits_{\ell\geq 0}2^{[
       a(1-\lambda)-s]\ell u}\Big(\sum\limits_{j\in
\N_0}2^{(j+\ell)s\theta}\Big[|P_{2^{j+\ell},a}f_{j+\ell}(t)|^{1-\lambda}M[|f_{j+\ell}|^{\lambda}
    ](t)\Big]^{\theta}\Big)^{u/\theta}  \\
    &&\qquad\qquad\lesssim ~\Big(\sum\limits_{j\in
\N_0}2^{js\theta}\Big[|P_{2^{j},a}f_{j}(t)|^{1-\lambda}M[|f_{j}|^{\lambda}
        ](t)\Big]^{\theta}\Big)^{u/\theta} \,.
\eeqq
Plugging this in \eqref{9.41a} and using H\"older's inequality twice with $1/\lambda$ and $1/(1-\lambda)$ gives
\beqq
    &&\Big\|\Big(\sum\limits_{\ell\geq 0}\Big(\sum\limits_{j\in
\N_0}2^{js\theta}\Big[2^j\int_{-2^{-j}}^{2^{-j}}|\Delta^m_h(\wt
T^{\psi}_1f_{j+\ell},\cdot)|\,dh\Big]^{\theta}\Big)^{u/\theta}\Big)^{1/u}\Big\|_{L_p(\Omega')} \\
    &&\qquad\quad\lesssim~\Big\|\Big(\sum\limits_{j\in
    \N_0} 2^{js\theta } \big|P_{2^{j},a}f_j\big|^ {\theta}\Big)^{1/\theta}\Big\|_{L_p(\mathbb{T})}^{1-\lambda}
\cdot \Big\| \Big(\sum\limits_{j\in
    \N_0} \big[2^{js\lambda }   M \big(\big|
P_{2^{j},a}f_j\big|^{\lambda}\big)\big]^{\frac{\theta}{\lambda}}\Big)^{\frac{\lambda}{\theta}}\Big\|_{L_\frac{p}{
\lambda}(\mathbb{T})}.
    \eeqq
Applying the periodic versions of
Theorems \ref{peetremax} and \ref{feffstein} (note that
$\lambda<\min\{p,\theta\}$), see \cite[Prop.\ 3.2.4, Thm.\ 3.4.1]{ST}, we get 
\be
\Big\|\Big(\sum\limits_{\ell\geq 0}\Big(\sum\limits_{j\in
\N_0}2^{js\theta}\Big[2^j\int_{-2^{-j}}^{2^{-j}}|\Delta^m_h(\wt
T^{\psi}_1f_{j+\ell},\cdot)|\,dh\Big]^{\theta}\Big)^{u/\theta}\Big)^{1/u}\Big\|_{L_p(\Omega')} \ \lesssim\
\|f\|_{F^s_{p,\theta}(\mathbb{T})}.\label{2d}
\ee
From \eqref{1d} and \eqref{2d} we obtain the desired estimate. The proof is complete. \\
\qed

\subsection{Proof of Theorems \ref{pointmult-B} and \ref{pointmult-F}}

\noindent In order to prove Theorems \ref{pointmult-B} and \ref{pointmult-F} we need a multivariate
version of Lemma \ref{1dim}.
\begin{lem}\label{ddim}\rm Let $b = (b_1,...,b_d)>0$, $a>0$, $e\subset [d]$, $m \in \N$, $\psi \in C^k_0(\R^d)$ with
$k\geq m$ and $h =
(h_1,...,h_d) \in \R^d$. Let further $f\in
\mathcal{S}'(\R^d)$ with $\supp(\mathcal{F} f) \subset Q_{b}$, where
$$
  Q_{b}:=[-b_1,b_1]\times...\times [-b_d,b_d]\,.
$$
Then there exists a constant $C>0$ (independent of $f$, $b$
and $h$) such that
\begin{equation*}
|\Delta^{m,e}_h (\psi f,x)|\\
\leq C_{m,a,\psi}\prod\limits_{i\in e}\max\{1,|b_ih_i|^a\}\min\{1,|b_ih_i|^m\}
\cdot P_{b,a}^ef(x)
\end{equation*}
holds for all $x\in \R^d$. Here
\beqq
            P_{b,a}^ef(x) := \sup_{y_i \in \R,\, i\in e}
             \frac{|f(y)|}{\prod_{i\in e}(1+|b_i(x_i-y_i)|)^{a}}
\eeqq
denotes the counterpart of \eqref{petfefste} for $e \subset [d]$\,.
\end{lem}

\bproof The proof goes along the same lines of the proof of Lemma \ref{1dim},
see also \cite[Lem.\ 3.3.2]{Ul06}\,.
\eproof

\noindent Now we are in position to prove Theorems \ref{pointmult-B} and \ref{pointmult-F}.\\
\vskip 0cm
\noindent
{\bf Proof of Theorem \ref{pointmult-F}.} We will show that with the assumption in Theorem \ref{pointmult-F} we have
\begin{equation}\label{nonpervsper}
    \Big\|\Big(\sum\limits_{j\in \N_{0}^d}
2^{s|j|_1\theta}\mathcal{R}^{e(j)}_m(\wt T^{\psi}_d f,2^{-j},\cdot)^{\theta}\Big)^{
1/\theta } \Big\|_{L_p(\R^d)} \lesssim \|f\|_{\Fspt(\tor)}\quad,\quad f\in
\Fspt(\tor)\,,
\end{equation}
where $\mathcal{R}^{e(j)}_m(\psi f,2^{-j},\cdot)$ is defined in
\eqref{rectm}. 
Let us first prove the case
$\min\{p,\theta\}\leq 1$. We choose $0<\lambda<\min\{p,\theta\}$ and $a>1/\min\{p,\theta\}$ as in the proof of
Lemma \ref{univ2}. Using the decomposition of unity, 
$\supp(\psi) \subset \Omega$ and $\Omega'=\Omega+[-1,1]^d$ we get
\beqq 
    &&\Big\|\Big(\sum\limits_{j\in \N_0^d}2^{|j|_1s\theta}\mathcal{R}^{e(j)}_m( \wt T^{\psi}_d
f,2^{-j},\cdot)^{\theta}\Big)^{1/\theta}\Big\|_p\nonumber\\
    &&\qquad\qquad\qquad\leq~  \Big\|\sum\limits_{\ell\in \Z^d}\Big(\sum\limits_{j\in
\N_0^d}2^{|j|_1s\theta}\mathcal{R}^{e(j)}_m( \psi
f_{j+\ell},2^{-j},\cdot)^{\theta}\Big)^{1/\theta}\Big\|_{L_p(\Omega')}\,,\label{start3}
\eeqq
see \eqref{start2}. For simplicity we assume that $\theta\geq 1$. Note, that $p<1$ is still possible here. The
modifications in case $\theta<1$ are straight-forward, see the proof of Lemma \ref{univ2}. From
\eqref{rectm} we have for $x\in \R^d$
\beqq
2^{|j|_1s}\mathcal{R}^{e(j)}_m( \psi f_{j+\ell},2^{-j},x) 
 &\lesssim &2^{|j|_1s}\int_{\Gamma_j^{e(j)}} 
     \big|\Delta_{h}^{m,e(j)} (\psi f_{j+\ell},x)\big|\prod_{i\in e(j)}2^{j_i}dh_i\,,
\eeqq
where $\Gamma_j^{e(j)}$, $e_1(j,\ell)$ and $e_2(j,\ell)$ are defined above in \eqref{Gammaj} and \eqref{ejl}\,.
Lemma \ref{ddim} yields
\beq
\Delta_{h}^{m,e(j)} (\psi f_{j+\ell},x)& = &\Delta_{h}^{m,e_1(j,\ell)}  \circ \Delta_{h}^{m,e_2(j,\ell)} (\psi
f_{j+\ell},x)\nonumber\\
&\lesssim &\Big( \prod_{i\in e_2(j,\ell)} |2^{j_i+\ell_i}h_i|^m\Big) \Delta_{h}^{m,e_1(j,\ell)}
(P_{2^{j+\ell},a}^{e_2(j,\ell)}f_{j+\ell},x)\,.\label{d1}
\eeq
Applying Lemma \ref{ddim} again, see also \cite[Lemma 1.12]{TDiss}, but this
time with the components in $e_1(j,\ell)$, we get
\beqq
&&\big|\Delta_{h}^{m,e_1(j,\ell)} (P_{2^{j+\ell},a}^{e_2(j,\ell)}f_{j+\ell},x)\big|\\
& &\qquad\qquad= ~\big|\Delta_{h}^{m,e_1(j,\ell)} (P_{2^{j+\ell},a}^{e_2(j,\ell)}f_{j+\ell},x)\big|^{1-\lambda}\cdot
\big| \Delta_{h}^{m,e_1(j,\ell)} (P_{2^{j+\ell},a}^{e_2(j,\ell)}f_{j+\ell},x)\big|^{\lambda}\\
&&\qquad\qquad\lesssim  ~  \Big[\Big(\prod_{i\in e_1(j,\ell)} |2^{j_i+\ell_i}h_i|^{a}\Big)
P_{2^{j+\ell},a}^{e(j)}f_{j+\ell}(x)\Big]^{1-\lambda}\cdot \big| \Delta_{h}^{m,e_1(j,\ell)}
(P_{2^{j+\ell},a}^{e_2(j,\ell)}f_{j+\ell},x)\big|^{\lambda}\,.
\eeqq
Inserting this into \eqref{d1} and performing the integration yields
\beqq
&& 2^{|j|_1s}\mathcal{R}^{e(j)}_m( \psi f_{j+\ell},2^{-j},x)\\
&&\lesssim\,2^{|j|_1s}\Big(\prod_{i\in e_2(j,\ell)} 2^{\ell_im}\Big)\Big(\prod_{i\in e_1(j,\ell)} 2^{\ell_i
a(1-\lambda)} \Big) \big[P_{2^{j+\ell},a}^{e(j)}f_{j+\ell}(x)\big]^{1-\lambda}\cdot   M_{e_1(j,\ell)}\big(\big|
P_{2^{j+\ell},a}^{e_2(j,\ell)}f_{j+\ell}(x)\big|^{\lambda}\big)\\
&&\lesssim\, 2^{|j|_1s}\Big(\prod_{i\in e_2(j,\ell)} 2^{\ell_im}\Big)\Big(\prod_{i\in e_1(j,\ell)} 2^{\ell_i
a(1-\lambda)} \Big) \big[P_{2^{j+\ell},a}f_{j+\ell}(x)\big]^{1-\lambda}\cdot   M_{[d]} \big(\big|
P_{2^{j+\ell},a}f_{j+\ell}(x)\big|^{\lambda}\big)\,.
\eeqq
Observe that
\beqq
&& 2^{|j|_1s}\Big(\prod_{i\in e_2(j,\ell)} 2^{\ell_im}\Big)\Big(\prod_{i\in e_1(j,\ell)} 2^{\ell_i a(1-\lambda)} \Big)\\
&&\qquad\qquad\qquad=~ 2^{|j+\ell|_1s}\Big(\prod_{i\in e_2(j,\ell)} 2^{\ell_i(m-s)}\Big)\Big(\prod_{i\in e_1(j,\ell)}
2^{\ell_i [a(1-\lambda)-s]} \Big)\Big(\prod_{i\in e_0(j)}2^{-\ell_is}\Big)\\
&&\qquad\qquad\qquad\leq~2^{|j+\ell|_1s}\Big(\prod_{i\in e_2(j,\ell)} 2^{\ell_i(m-s)}\Big)\Big(\prod_{i\in e_1(j,\ell)\cup 
 e_0(j)}2^{-\ell_i\delta}\Big)\\
&&\qquad\qquad\qquad=~2^{|j+\ell|_1s}\Big(\prod_{i:\,\ell_i<0} 2^{\ell_i(m-s)}\Big)\Big(\prod_{i:\,\ell_i\geq
0}2^{-\ell_i\delta}\Big)\,,
\eeqq
where $\delta =\min\{s-a(1-\lambda),s\}>0$. This allows us to estimate
\beqq
&&\sum\limits_{\ell\in \Z^d}\Big(\sum\limits_{j\in
\N_0^d}2^{|j|_1s\theta}\mathcal{R}^{e(j)}_m( \psi f_{j+\ell},2^{-j},\cdot)^{\theta}\Big)^{1/\theta}\\
&&\qquad\qquad\qquad\qquad \lesssim~ \Big(\sum\limits_{j\in
\N_0^d} 2^{|j|_1s\theta } \Big[\big|P_{2^{j},a}f_j(x)\big|^{1-\lambda}\cdot   M_{[d]} \big(\big|
P_{2^{j},a}f_j(x)\big|^{\lambda}\big)\Big]^{\theta}\Big)^{1/\theta}\,.
\eeqq
Applying H\"older's inequality twice with the pair $(1/(1-\lambda),1/\lambda)$ and use afterwards the periodic
version of Theorem \ref{feffstein}, see \cite[Thm.\ 4.1.2]{Ul06} and \cite[Prop.\ 1.3]{TDiss}, we arrive at 
\beqq
 &&\Big\|\sum\limits_{\ell\in \Z^d}\Big(\sum\limits_{j\in
 \N_0^d}2^{|j|_1s\theta}\mathcal{R}^{e(j)}_m( \psi
f_{j+\ell},2^{-j},\cdot)^{\theta}\Big)^{1/\theta}\Big\|_{L_p(\Omega')}\\
&&~\lesssim~\Big\|\Big(\sum\limits_{j\in
\N_0^d} 2^{|j|_1s\theta } \big|P_{2^{j},a}f_j\big|^ {\theta}\Big)^{1/\theta}\Big\|_{L_p(\tor)}^{1-\lambda} \cdot
\Big\| \Big(\sum\limits_{j\in
\N_0^d} \big[2^{|j|_1s\lambda }   M_{[d]} \big(\big|
P_{2^{j},a}f_j\big|^{\lambda}\big)\big]^{\frac{\theta}{\lambda}}\Big)^{\frac{\lambda}{\theta}}\Big\|_{L_\frac{p}{
\lambda}(\tor)}\\
& &~\lesssim~\Big\|\Big(\sum\limits_{j\in
\N_0^d} 2^{|j|_1s\theta } \big|P_{2^{j},a}f_j\big|^ {\theta}\Big)^{1/\theta}\Big\|_{L_p(\tor)}\,,
\eeqq
since the function inside the $L_p$-(quasi-)norm is $1$-periodic. The periodic version of Theorem \ref{peetremax}, see
\cite[Thm.\ 4.1.3]{Ul06} and \cite[Prop.\ 1.5]{TDiss}, implies the inequality
\eqref{nonpervsper}. The case $\min\{p,\theta\}>1$ is less technical. We choose $\lambda=1$. Next we take the
integral in \eqref{d1} and then estimate above by the Hardy-Littlewood maximal function according to the directions in
$e_1(j,\ell)$. The rest is carried out similar to the case $\min\{p,\theta\}\leq 1$. The proof is complete.\\
\qed
\vskip 0.5cm
\begin{rem} The proof of Theorem \ref{pointmult-B} (${\mathbf B}$-case) is similar but
technically less involved. As in the ${\mathbf F}$-case we choose $0<\lambda<p$ and $a>1/p$ such that
$s-a(1-\lambda)>0$ if $p\leq 1$ otherwise we put $\lambda=1$. Again, one can do so because of $s>\sigma_p$. Note, that
the weaker condition $s>\sigma_p$ comes
from the fact that the Peetre maximal function \eqref{petfefste} is bounded in
$L_p(\tor)$, $0<p<\infty$, if $a>1/p$, which is the scalar version of Theorem \ref{peetremax}\,.
\end{rem}

\subsection*{Acknowledgments}
\begin{sloppypar}
The authors acknowledge the fruitful discussions with D. Bazarkhanov, W. Sickel, V.N. Temlyakov, H. Triebel, and Dinh
D\~ung on this topic, especially at the ICERM Semester Programme ``High-Dimensional Approximation'' in Providence, 2014,
where this project has been initiated, and at the conference ``Approximation Methods and Function Spaces'' in
Hasenwinkel, 2015. The authors would particularly thank D. Bazarkhanov for pointing out reference \cite{Du2} to us.
Moreover, Tino Ullrich gratefully acknowledges support by the German Research Foundation (DFG) and the Emmy-Noether
programme, Ul-403/1-1.
\end{sloppypar}


\end{document}